\documentclass[12pt]{article}

\title{Convergence of implicit schemes for Hamilton-Jacobi-Bellman
quasi-variational inequalities}

\author{%
  Parsiad Azimzadeh%
  \thanks{Department of Mathematics, University of Michigan
  ({\tt parsiad@umich.edu}, {\tt erhan@umich.edu}).}
  \and
  Erhan Bayraktar\footnotemark[1] %
  \thanks{E. Bayraktar is supported in part by the National Science
  Foundation under grant DMS-1613170.}
  \and
  George Labahn%
  \thanks{David R. Cheriton School of Computer Science, University of
  Waterloo ({\tt glabahn@uwaterloo.ca}).}%
}

\usepackage{algorithmic}
\usepackage{amsfonts}
\usepackage{amssymb}
\usepackage{booktabs}
\usepackage{enumitem}
\usepackage{mathtools}
\usepackage{nicefrac}
\usepackage{subfig}
\usepackage{multicol}
\usepackage{verbatim}

\newcommand{\loc}{\operatorname{loc}}
\newcommand{\const}{\operatorname{const.}}
\newcommand{\trace}{\operatorname{trace}}
\newcommand{\interp}{\operatorname{interp}}

\newcommand{\mathand}{\qquad\text{and}\qquad}
\newcommand{\mathas}{\qquad\text{as}\qquad}
\newcommand{\mathwhere}{\qquad\text{where }}

\newcommand{\mathfor}{\qquad\text{for }}

\makeatletter
\@ifclassloaded{siamart1116}{

  \newsiamthm{example}{Example}
  \newsiamthm{rem}{Remark}

  \headers{Implicit schemes for HJBQVIs}{Parsiad Azimzadeh, Erhan
  Bayraktar, and George Labahn}

  \bibliographystyle{plainnat}
  \newcommand{\qedhere}{}

}{

  \def\myarticle{1}

  \date{}

  \pdfoutput=1

  \usepackage{microtype}
  \usepackage[T1]{fontenc}
  \usepackage{lmodern}

  \usepackage{amsthm}
  \usepackage[numbers]{natbib}

  \usepackage[margin=1in]{geometry}

  \renewcommand*{\leq}{\leqslant}
  \renewcommand*{\geq}{\geqslant}

  \usepackage{titlesec}

  \newcounter{dummy}
  \newtheorem{definition}[dummy]{Definition}
  
  \newtheorem{lemma}[dummy]{Lemma}
  \newtheorem{proposition}[dummy]{Proposition}
  \newtheorem{rem}[dummy]{Remark}
  \newtheorem{theorem}[dummy]{Theorem}
  \newenvironment{keywords}{{\bf Key words.}}{\medskip{}}
  \newenvironment{AMS}{{\bf AMS subject classifications.}}{}

  \usepackage{hyperref}
  \bibliographystyle{plainnat}

  \usepackage{prettyref}
  \newcommand{\cref}{\prettyref}
  \newcommand{\Cref}{\prettyref}
  \newrefformat{def}{Definition \ref{#1}}
  \newrefformat{thm}{Theorem \ref{#1}}
  \newrefformat{app}{Appendix \ref{#1}}
  \newrefformat{rem}{Remark \ref{#1}}
  \newrefformat{sec}{Section \ref{#1}}
  \newrefformat{subsec}{Section \ref{#1}}
  \newrefformat{subsubsec}{Section \ref{#1}}
  \newrefformat{alg}{Algorithm \ref{#1}}
  \newrefformat{fig}{Figure \ref{#1}}
  \newrefformat{tab}{Table \ref{#1}}
  \newrefformat{prop}{Proposition \ref{#1}}

  \newcommand*{\email}[1]{\href{mailto:#1}{\nolinkurl{#1}} }

  \usepackage[dvipsnames]{xcolor}
  \definecolor{mylinkcolor}{HTML}{0066cc}
  \definecolor{mycitecolor}{HTML}{008800}
  \definecolor{myurlcolor}{HTML}{0066cc}

  \hypersetup{
    linkcolor  = mylinkcolor,
    citecolor  = mycitecolor,
    urlcolor   = myurlcolor,
    colorlinks = true,
  }

  \providecommand{\doi}{}
  \renewcommand{\doi}[1]{\href{https://doi.org/#1}{doi:
  \begingroup\urlstyle{rm}\nolinkurl{#1}\endgroup}}

}
\makeatother

\usepackage{tikz}
\usepackage{todonotes}

\begin{document}

\maketitle

\begin{abstract}
  In [Azimzadeh, P., and P. A. Forsyth. ``Weakly chained matrices,
    policy iteration, and impulse control.'' \emph{SIAM J. Num. Anal.}
  54.3 (2016): 1341-1364], we outlined the theory and implementation
  of computational methods for implicit schemes for
  Hamilton-Jacobi-Bellman quasi-variational inequalities (HJBQVIs).
  No convergence proofs were given therein.
  This work closes the gap by giving rigorous proofs of convergence.
  We do so by introducing the notion of \emph{nonlocal consistency}
  and appealing to a Barles-Souganidis type analysis.
  Our results rely only on a well-known comparison principle and are
  independent of the specific form of the intervention operator.
\end{abstract}



\section{Introduction}

We consider numerical methods for approximating the viscosity
solution of the Hamilton-Jacobi-Bellman quasi-variational inequality (HJBQVI)
\begin{equation}
  \begin{cases}
    \displaystyle{
      \min \left (
        - \sup_{b \in B} \left \{
          \frac{\partial u}{\partial t} + L_b u + f(\cdot, b)
        \right \},
        u - \mathcal{M} u
      \right ) = 0,
    }
    & \text{on } [0, T) \times \mathbb{R}^d
    \\
    \displaystyle{
      u(T, \cdot) = \max \left ( g, \mathcal{M} u(T, \cdot) \right ),
    }
    & \text{on } \mathbb{R}^d
  \end{cases}
  \label{eq:hjbqvi}
\end{equation}
where $L_b$ is the second order operator
\begin{equation}
  L_b u(t, x)
  \coloneqq \left \langle \mu(x, b), D_x u(t, x) \right \rangle
  + \frac{1}{2} \trace( \sigma \sigma^\intercal(x, b) D_x^2 u(t, x) )
  \label{eq:generator}
\end{equation}
and $\mathcal{M}$ is the so-called ``intervention'' operator
\begin{equation}
  \mathcal{M} u(t, x)
  \coloneqq \sup_{z \in Z(t, x)} \left \{
    u(t, x + \Gamma(t, x, z)) + K(t, x, z)
  \right \}.
  \label{eq:intervention}
\end{equation}
In \cref{eq:generator}, $D_x u$ and $D_x^2 u$ are the gradient vector
and second derivative matrix of $u$ with respect to the spatial variable $x$.
In \cref{eq:intervention}, $\cup_{t, x} Z(t, x)$ is a subset of some
metric space $Y$ and $\Gamma$ and $K$ are maps from $[0, T] \times
\mathbb{R}^d \times Y$ to $\mathbb{R}^d$ and $\mathbb{R}$, respectively.

It is well-known (see, e.g., \cite[Chapter 7]{MR2109687}) that the
solution of the HJBQVI can be obtained as the limit of the sequence
$(u^k)_k$ where $u^0 = -\infty$ and, for $k \geq 1$,
\[
  \begin{cases}
    \displaystyle{
      \min \left (
        - \sup_{b \in B} \left \{
          \frac{\partial u^k}{\partial t} + L_b u^k + f(\cdot, b)
        \right \},
        u^k - \mathcal{M} u^{k-1}
      \right ) = 0,
    }
    & \text{on } [0, T) \times \mathbb{R}^d
    \\
    \displaystyle{
      u^k(T, \cdot) = \max \left ( g, \mathcal{M} u^{k-1}(T, \cdot) \right ),
    }
    & \text{on } \mathbb{R}^d.
  \end{cases}
\]
The above is a variational (not quasi-variational) inequality since
the obstacle $\mathcal{M} u^{k - 1}$ does not depend on the solution $u^k$.
This approach, referred to as iterated optimal stopping, suggests a
simple algorithm for solving the HJBQVI numerically: fix mesh sizes
$\Delta t$ and $\Delta x$ and compute a finite difference
approximation to $u^1$, use that to compute an approximation to
$u^2$, etc. until convergence up to a desired tolerance.
Unfortunately, this approach suffers from a high space complexity
\cite{MR3150265}.

In light of this, a previous work of the first author
\cite{MR3493959} considers various schemes which do not resort to a
variational approximation as per the previous paragraph.
In \cite{MR3493959}, the focus is on the theory and implementation of
computational methods for solving these schemes; no convergence
proofs are given.
In this work, we close the gap by giving rigorous convergence results.
The schemes in \cite{MR3493959} are ``implicit'' in the sense that
second derivative terms are approximated using information from the
current timestep, ensuring unconditional stability.
This is in contrast to so-called ``explicit'' schemes, which
approximate second derivative terms using information from the
previously computed timestep, requiring a harsh timestep-size
restriction of the form $\Delta t = \const (\Delta x)^2$ to ensure
stability (see, e.g., \cite[Chapter 4]{MR1217486}).

Our closest related works \cite{MR2407322,MR2597608} introduce
implicit (in the sense above) finite difference schemes to solve
HJBQVIs arising from insurance applications.
However,
\begin{itemize}
  \item
    \cite{MR2407322,MR2597608} assume a strong form of comparison
    principle for the HJBQVI which, to the best of our knowledge,
    does not exist in the literature.
    In contrast, we do not make any such assumptions, relying only on
    a well-known comparison principle (see \cref{app:comparison_principle}).
  \item
    Moreover, \cite{MR2407322,MR2597608} assume a specific form for
    the intervention operator $\mathcal{M}$ and as such, the
    convergence proofs do not generalize.
    For example, \cite{MR2407322} considers a two-dimensional problem
    ($x = (x_1, x_2)$) with
    \[
      \mathcal{M} u(t, x)
      = \sup_{z \in [0, x_2]} \left \{
        u(t, \max \{x_1 - z, 0 \}, x_2 - z) + \kappa z - c
      \right \}
    \]
    where $\kappa$ and $c$ are constants.
    In contrast, we do not assume specific forms for the functions
    $Z$, $\Gamma$, and $K$ in \cref{eq:intervention} (see, in
    particular, \cref{lem:intervention_limits}).
\end{itemize}

A significant part of focus is on the first issue, which is subtle in nature.
To briefly sketch the issue, we first rewrite the HJBQVI in a more
abstract form.
Namely, let
\begin{multline}
  F( (t, x), r, (a, p), A, \ell ) \coloneqq
  \\
  \begin{cases}
    \min (
      - \sup_{b \in B} \{
        a
        {+} \langle \mu(x, b), p \rangle
        {+} \frac{1}{2} \operatorname{trace}(\sigma \sigma^\intercal(x, b) A)
        {+} f(t, x, b)
      \},
      r - \ell
    ),
    & \text{if } t < T
    \\
    \min (
      r - g(x),
      r - \ell
    ),
    & \text{if } t = T
  \end{cases}
  \label{eq:hjbqvi_operator}
\end{multline}
where we have used the notation $(a, p)$ to distinguish between the
time and spatial derivatives (i.e., $a = \partial u / \partial t(t,
x)$ and $p = D_x u(t, x)$) and $A$ to refer to the second spatial
derivative (i.e., $A = D_x^2 u(t, x)$).
Letting $D \coloneqq (\partial u / \partial t, D_x)$, \cref{eq:hjbqvi} becomes
\begin{equation}
  F( y, u(y), D u(y), D_x^2 u(y), \mathcal{M} u(y) ) = 0,
  \mathfor y = (t, x) \in [0, T] \times \mathbb{R}^d.
  \label{eq:nonlocal}
\end{equation}
Now, there are a few reasonable notions of viscosity solution for
\cref{eq:nonlocal}.
Roughly speaking, one notion is obtained by replacing derivatives of
$u$ with those of a test function $\varphi$:
\begin{equation}
  F(y, u(y), D \varphi(y), D_x^2 \varphi(y), \mathcal{M} u(y)) = 0.
  \tag{def.1} \label{eq:nonlocal_definition}
\end{equation}
Another is obtained by replacing instead \emph{all} terms involving $u$:
\begin{equation}
  F(y, \varphi(y), D \varphi(y), D_x^2 \varphi(y), \mathcal{M} \varphi(y)) = 0.
  \tag{def.2} \label{eq:local_definition}
\end{equation}
Since the operator $\mathcal{M}$ is nondecreasing (i.e., $\mathcal{M}
u \leq \mathcal{M} w$ whenever $u \leq w$), an upper semicontinuous
(USC) subsolution in the sense of \cref{eq:nonlocal_definition} is
also a USC subsolution in the sense of \cref{eq:local_definition}.
The same is true for lower semicontinuous (LSC) supersolutions and as
such, a comparison principle in the sense of
\cref{eq:local_definition} implies a comparison principle in the
sense of \cref{eq:nonlocal_definition} (see
\cref{sec:discontinuous_solutions} for details).

As hinted at above, the results of \cite{MR2407322,MR2597608} require
the stronger \cref{eq:local_definition}-comparison principle which,
to our knowledge, does not exist in the literature for the HJBQVI.
To be able to rely only on a well-known
\cref{eq:nonlocal_definition}-comparison principle, we introduce the
notion of \emph{nonlocal consistency}.
Analogously to the Barles-Souganidis framework (BSF) for local
equations \cite{MR1115933}, we establish that a monotone, stable, and
nonlocally consistent scheme whose limiting equation satisfies a
\cref{eq:nonlocal_definition}-comparison principle is convergent.
We use this to prove the convergence of the implicit schemes from
\cite{MR3493959}.

This work is organized as follows.
\Cref{sec:discontinuous_solutions} discusses the different notions of
viscosity solutions and their relationship with the BSF.
\Cref{sec:convergence_result} introduces the notion of nonlocal
consistency and gives the BSF-type convergence result described in
the previous paragraph.
\Cref{sec:hjbqvi} applies this result to prove the convergence of
various schemes for the HJBQVI.

\section{Viscosity solutions of nonlocal equations}
\label{sec:discontinuous_solutions}

In this and the following section, instead of considering the HJBQVI
directly, we consider the more general PDE
\begin{equation}
  F(x, u(x), Du(x), D^2 u(x), \mathcal{M} u(x)) = 0
  \text{ for } x \in \overline{\Omega}
  \label{eq:pde}
\end{equation}
where $\Omega \subset \mathbb{R}^d$ is open, $\overline{\Omega}$ is
its closure, $u$ is in $B_{\loc}(\overline{\Omega})$ (the set of
locally bounded real-valued maps from $\overline{\Omega}$), and $D u$
and $D^2 u$ are the formal gradient vector and second derivative matrix of $u$.
Note that no generality is lost in using $x$ instead of $(t, x)$ in
\cref{eq:pde} since a ``time'' variable $t$ can be incorporated by
increasing the dimensionality $d$ of the space.
With this in mind, it is clear that the HJBQVI is a special case of
\cref{eq:pde} which is obtained when $F$ and $\mathcal{M}$ are
defined as per \cref{eq:intervention} and \cref{eq:hjbqvi_operator}.

In our analyses, we will always assume that $\mathcal{M}$ maps
$B_{\loc}(\overline{\Omega})$ to some subset of itself and that $F$
is a locally bounded real-valued map from $\overline{\Omega} \times
\mathbb{R} \times \mathbb{R}^d \times \mathcal{S}^d \times
\mathbb{R}$ that is \emph{elliptic} in the sense of \cite[Assumption
(E)]{MR2422079}:
\[
  F(x, r, p, A, \ell_1) \geq F(x, r, p, B, \ell_2)
  \text{ for } A \preceq B \text{ and } \ell_1 \leq \ell_2
\]
where $\preceq$ is the usual semidefinite order on $\mathcal{S}^d$
(the real $d \times d$ symmetric matrices).
To stress the significance of the operator $\mathcal{M}$, we call
\cref{eq:pde} a \emph{nonlocal second order equation}.


For a locally bounded real-valued function $z$ mapping from a metric
space, we define its upper and lower semicontinuous envelopes by
$z^*(x) \coloneqq \limsup_{y \rightarrow x} z(y)$ and $z_* \coloneqq
-(-z)^*$ where $\limsup_{y \rightarrow x} z(y) \coloneqq
\lim_{\epsilon \downarrow 0} (\sup \{ z(y) \colon |y - x| < \epsilon \} )$.
We are now ready to state the definition of viscosity solution hinted
at in \cref{eq:nonlocal_definition}.

\begin{definition}
  \label{def:viscosity_solution}
  $u \in B_{\loc}(\overline{\Omega})$ is a (viscosity) subsolution
  (resp. supersolution) of \cref{eq:pde} if for all $\varphi \in
  C^{2}(\overline{\Omega})$ and $x \in \overline{\Omega}$ such that
  $u^* - \varphi$ (resp. $u_* - \varphi$) has a local maximum (resp.
  minimum) at $x$, we have
  \begin{align*}
    & F_*(x, u^*(x), D \varphi(x), D^2 \varphi(x), \mathcal{M} u^*(x)) \leq 0 \\
    \text{(resp. }
      & F^*(x, u_*(x), D \varphi(x), D^2 \varphi(x), \mathcal{M} u_*(x)) \geq 0
    \text{)}.
  \end{align*}
  $u$ is said to be a (viscosity) solution of \cref{eq:pde} if it is
  both a sub and supersolution of \cref{eq:pde}.
\end{definition}

Our usage of $F_*$ and $F^*$ above is so that we may write both the
partial differential equation and its boundary conditions as a single
expression (for a detailed explanation, see \cite[Pg. 274]{MR1115933}).
We now state the definition of viscosity solution hinted at in
\cref{eq:local_definition}.

\begin{definition}
  \label{def:viscosity_solution_2}
  $u \in B_{\loc}(\overline{\Omega})$ is a subsolution of
  \cref{eq:pde} if for all $\varphi \in C^{2}(\overline{\Omega})$ and
  $x \in \overline{\Omega}$ such that $(u^* - \varphi)(x) = 0$ is a
  global maximum of $u^* - \varphi$, we have
  \[
    F_*(x, \varphi(x), D \varphi(x), D^2 \varphi(x), \mathcal{M}
    \varphi(x)) \leq 0.
  \]
  Supersolutions are defined symmetrically.
\end{definition}

Throughout this article, it should be assumed that unless otherwise
specified, the terms subsolution/supersolution/solution refer to the
concepts in \cref{def:viscosity_solution}.
When we wish to be explicit, we will write
\cref{eq:nonlocal_definition}-subsolution/supersolution/solution or
\cref{eq:local_definition}-subsolution/supersolution/solution.

If the operator $\mathcal{M}$ is nondecreasing (i.e., $\mathcal{M} u
\leq \mathcal{M} w$ whenever $u \leq w$), then any
\cref{eq:nonlocal_definition}-subsolution (resp. supersolution) is
trivially a \cref{eq:local_definition}-subsolution (resp. supersolution).
To see this, let $u$ be a \cref{eq:nonlocal_definition}-subsolution,
$\varphi \in C^2(\overline{\Omega})$, and $x \in \overline{\Omega}$
be such that $(u^* - \varphi)(x) = 0$ is a global maximum of $u^* - \varphi$.
In this case, for any $y \in \overline{\Omega}$,
\[
  0
  = (u^* - \varphi)(x)
  \geq (u^* - \varphi)(y).
\]
Therefore, $\varphi \geq u^*$, from which we obtain $\mathcal{M}
\varphi \geq \mathcal{M} u^*$.
Using the ellipticity of $F$,
\[
  F_*(x, \varphi(x), D \varphi(x), D^2 \varphi(x), \mathcal{M} \varphi(x))
  \leq F_*(x, u^*(x), D \varphi(x), D^2 \varphi(x), \mathcal{M} u^*(x))
  \leq 0.
\]

As for the converse, a \cref{eq:local_definition}-subsolution (resp.
supersolution) need not be a
\cref{eq:nonlocal_definition}-subsolution (resp. supersolution), even
if $\mathcal{M}$ is nondecreasing.
This is because the term $\mathcal{M} \varphi(x)$ is not determined
by values of $\varphi$ in a neighbourhood of $x$.
In order to control this term, we need to impose some additional continuity:

\begin{proposition}
  Suppose the operator $\mathcal{M}$ is nondecreasing and that there
  exists a function $\omega : [0, \infty] \rightarrow [0, \infty]$
  satisfying $\lim_{\epsilon \downarrow 0 } \omega(\epsilon) = 0$ and
  \[
    \mathcal{M}(v+\epsilon)
    \leq \mathcal{M}v + \omega(\epsilon)
    \qquad
    \text{(resp. } \mathcal{M}(v - \epsilon)
      \geq \mathcal{M}v - \omega(\epsilon)
    \text{)}
  \]
  for all $v \in C(\overline{\Omega})$ and $\epsilon > 0$.
  If $u$ is a \underline{continuous}
  \cref{eq:local_definition}-subsolution (resp. supersolution), it is
  also a \cref{eq:nonlocal_definition}-subsolution (resp. supersolution).
\end{proposition}

The above generalizes \cite[Theorem 3.1]{MR2486085},
\cite[Proposition 1.2]{MR784578}, and \cite[Section 4]{MR2735526}.
We point out that the requirement involving the modulus of continuity
$\omega$ is rather innocuous.
For example, when $\mathcal{M}$ is the intervention operator from the
introduction, the choice of $\omega(\epsilon) = \epsilon$ satisfies
the requirement under mild conditions (e.g., $Z(t, x)$ is nonempty
  and compact for each $(t, x)$ and $\Gamma$ and $K$ are continuous;
  see also assumptions \ref{enu:vanilla_start} and
\ref{enu:vanilla_middle} of the sequel).

\begin{proof}
  Let $u$ be a continuous \cref{eq:local_definition}-subsolution (we
  omit the supersolution case, which is handled similarly).
  Let $\varphi \in C^2(\overline{\Omega})$ and $x \in
  \overline{\Omega}$ be such that $u - \varphi$ has a local maximum at $x$.
  Without loss of generality, we may assume $u(x) = \varphi(x)$.
  Let $\epsilon > 0$.
  We claim that, by the continuity of $u$, we can construct a smooth
  function $\psi$ such that $u \leq \psi \leq u + \epsilon$
  everywhere and $\psi = \varphi$ on a neighbourhood of $x$.
  In this case,
  \[
    \mathcal{M} u
    \leq \mathcal{M} \psi
    \leq \mathcal{M}(u + \epsilon)
    \leq \mathcal{M}u + \omega(\epsilon)
    \eqqcolon \mathcal{M}_\epsilon u
  \]
  and hence
  \[
    F_*(x, u(x), D \varphi(x), D^2 \varphi(x), \mathcal{M}_\epsilon u(x))
    \leq F_*(x, \psi(x), D \psi(x), D^2 \psi(x), \mathcal{M} \psi(x))
    \leq 0.
  \]
  Taking limit inferiors of this inequality with respect to
  $\epsilon$ establishes the desired result.

  We now return to the claim above.
  To simplify notation, we can, without loss of generality, assume
  $u$ and $\varphi$ are maps from $\mathbb{R}^d$. 
  By Lindel\"{o}f's lemma, we can find a countable cover $\{ U_n
  \}_n$ of $\mathbb{R}^d$ by open balls $U_n$ centred at points $x_n$
  and satisfying $|u(x_n) - u(\cdot)| < \epsilon / 2$ on $U_n$.
  By virtue of this, we can also find a smooth partition of unity $\{
  \chi_n \}_n \subset C^{\infty}(\mathbb{R}^d)$ subordinate to $\{ U_n \}_n$.
  Let
  \[
    \chi(\cdot) = \sum_n c_n \chi_n(\cdot)
    \mathwhere
    c_n \coloneqq \sup u(U_n).
  \]
  Now, fix $y$ and let $\mathcal{N} \coloneqq \{ n \colon \chi_n(y) \neq 0 \}$.
  Then, since $\mathcal{N}$ is a finite set,
  \[
    \chi(y)
    = \sum_n c_n \chi_n(y)
    \geq \left( \min_{n \in \mathcal{N}} c_n \right) \sum_n \chi_n(y)
    = \min_{n \in \mathcal{N}} c_n
    \geq u(y).
  \]
  Similarly, $\chi(y) \leq \max_{n \in \mathcal{N}} c_n = c_m$ for
  some $m \in \mathcal{N}$.
  Since $\chi_m(y) \neq 0$ implies $y \in U_m$, we have
  \[
    c_m
    \leq u(x_m) + \frac{\epsilon}{2}
    \leq u(y) + \epsilon.
  \]
  Since $y$ was arbitrary, this establishes $u \leq \chi \leq u + \epsilon$.

  Now, pick $r > 0$ small enough so that $u \leq \varphi \leq u +
  \epsilon$ on $\overline{B(x, 2r)}$, the closed ball centred at $x$
  with radius $2r$.
  Let $\zeta$ be a smooth cutoff function satisfying $0 \leq \zeta
  \leq 1$, $\zeta = 1$ on $\overline{B(x, r)}$, and $\zeta = 0$
  outside of $B(x, 2r)$.
  Then, the function $\psi$ defined by
  \[
    \psi(y) \coloneqq \zeta(y) \varphi(y) + \left(1 - \zeta(y)\right) \chi(y)
  \]
  satisfies our requirements.
\end{proof}

We close this section by discussing the BSF.
In order to do so, we first recall the notion of a comparison principle.
Let $U$ be some subset of $B_{\loc}(\overline{\Omega})$.
For $k=1,2$, we say \cref{eq:pde} satisfies a (def.$k$)-comparison
principle in $U$ whenever the following holds:
\[
  \text{if } u, w \in U
  \text{ are a (def.} k \text{)-subsolution}
  \text{ and (def.} k \text{)-supersolution of }
  \cref{eq:pde} \text{, then } u^* \leq w_*.
\]
The discussion preceding this paragraph implies that if $\mathcal{M}$
is nondecreasing, a \cref{eq:local_definition}-comparison principle
implies a \cref{eq:nonlocal_definition}-comparison principle (but not
vice versa).

The BSF provides a general approach for proving the convergence of
finite difference schemes to the viscosity solution $u$ of a PDE
\cite{MR1115933}.
The idea is as follows: letting $u^{\Delta x} : \overline{\Omega}
\rightarrow \mathbb{R}$ denote the solution of a finite difference
scheme (for a fixed grid size $\Delta x > 0$), the BSF outlines three
sufficient conditions (\emph{monotonicity}, \emph{stability}, and
\emph{consistency}) which guarantee the functions $\overline{u}$ and
$\underline{u}$ defined by
\[
  \overline{u}(x) \coloneqq \limsup_{\substack{
      \Delta x \downarrow 0 \\
      y \rightarrow x
  }} u^{\Delta x}(y)
  \mathand
  \underline{u}(x) \coloneqq \liminf_{\substack{
      \Delta x \downarrow 0 \\
      y \rightarrow x
  }} u^{\Delta x}(y)
\]
to be a subsolution/supersolution pair of the limiting equation.
Applying a comparison principle, one concludes $\overline{u} =
\underline{u} = u$.
In our context, the subtlety here is that if $\overline{u}$ and
$\underline{u}$ are a
\cref{eq:local_definition}-subsolution/supersolution pair, the above
argument requires a \cref{eq:local_definition}-comparison principle.

More generally, the BSF characterizes when a family of approximations
$(u^\rho)_{\rho > 0}$ converges to $u$ as $\rho \downarrow 0$ (in the
example above, $\rho = \Delta x$ was the grid size).
The BSF represents an approximation $u^\rho$ as a solution of the equation
\begin{equation}
  S(\rho, x, u^\rho(x), [u^\rho]_x) = 0
  \text{ for } x \in \overline{\Omega}.
  \label{eq:barles_scheme}
\end{equation}
The function $S$ is referred to as the approximation scheme.
Given a function $w$, the symbol $[w]_x$ (cf. \cite{MR1916291})
denotes a function which agrees with $w$ everywhere except possibly at $x$:
\[
  [w]_x(y) \coloneqq
  \begin{cases}
    w(y) & \text{if } y \neq x \\
    0    & \text{otherwise}.
  \end{cases}
\]
While the BSF is posed in terms of local equations of the form
\[
  F(x, u(x), D u(x), D^2 u(x)) = 0,
\]
the authors of \cite{MR2407322,MR2597608} (see, in particular,
\cite[Lemma 5.2]{MR2407322} and \cite[Definition 6.11]{MR2597608})
noticed that the BSF could also be applied to nonlocal equations of
the form \cref{eq:pde} by changing the consistency condition of the
BSF \cite[Eq (2.4)]{MR1115933} to
\begin{subequations}
  \begin{align}
    \limsup_{\substack{
        \rho \downarrow 0 \\
        y \rightarrow x \\
        \xi \rightarrow 0
    }}
    S(\rho, y, \varphi(y) + \xi, [\varphi + \xi]_y)
    & \leq
    F^*(x, \varphi(x), D \varphi(x), D^2 \varphi(x), \mathcal{M} \varphi(x))
    \\
    \liminf_{\substack{
        \rho \downarrow 0 \\
        y \rightarrow x \\
        \xi \rightarrow 0
    }}
    S(\rho, y, \varphi(y) + \xi, [\varphi + \xi]_y)
    & \geq
    F_*(x, \varphi(x), D \varphi(x), D^2 \varphi(x), \mathcal{M} \varphi(x)).
  \end{align}
  \label{eq:local_consistency}%
\end{subequations}
The issue with this approach, however, is that it requires a
\cref{eq:local_definition}-comparison principle.
In \cite{MR2407322,MR2597608}, the authors avoid the issue by simply
assuming that such a comparison principle holds (see, in particular,
  \cite[Assumption 5.1]{MR2407322} and \cite[Assumption
6.3]{MR2597608}), though it is not clear if such an assumption is
justifiable for HJBQVIs.

We show in the next section that an alternative notion of consistency
can be used so that only a \cref{eq:nonlocal_definition}-comparison
principle is required.
Ultimately, we will use this to prove the convergence of various
finite difference schemes for HJBQVIs, relying only on a well-known
\cref{eq:nonlocal_definition}-comparison principle given in
\cref{app:comparison_principle}.

\section{Convergence result}
\label{sec:convergence_result}

In our setting, approximation schemes take the form
\begin{equation}
  S(\rho, x, u^\rho(x), [u^\rho]_x, \mathcal{N}^\rho u^\rho(x)) = 0
  \text{ for } x \in \overline{\Omega}
  \label{eq:scheme}
\end{equation}
(compare with \cref{eq:barles_scheme}).
In particular, $S$ is a real-valued map from $(0, \infty) \times
\overline{\Omega} \times \mathbb{R} \times B(\overline{\Omega})
\times \mathbb{R}$, $\mathcal{N}^{\rho}$ maps $B(\overline{\Omega})$
to some subset of itself, and $B(\overline{\Omega})$ is the set of
bounded real-valued maps from $\overline{\Omega}$.
Intuitively, $\mathcal{N}^\rho$ will serve as an approximation of the
operator $\mathcal{M}$.
We refer to a function $u^\rho$ in $B(\overline{\Omega})$ as a
solution of the scheme if it satisfies \cref{eq:scheme}.

For brevity, let $\mathbb{R}_+ \coloneqq (0, \infty)$.
We will always assume the existence of functions $h_1 : \mathbb{R}_+
\rightarrow \mathbb{R}_+$ and $h_2 : \mathbb{R}_+ \times
\overline{\Omega} \rightarrow \mathbb{R}$ satisfying
\begin{equation}
  \lim_{
    \rho \downarrow 0
  } h_1(\rho)
  = 0
  \mathand
  \lim_{\substack{
      \rho \downarrow 0 \\
      y \rightarrow x
  }}
  h_2(\rho, y)
  = 0
  \label{eq:margin_assumption_1}
\end{equation}
along with
\begin{equation}
  \left |
  S(\rho, x, r \pm h_1(\rho), u, \ell)
  - S(\rho, x, r, u, \ell)
  \right |
  \leq h_2(\rho, x)
  \text{ for all } \rho, x, r, u, \ell.
  \label{eq:margin_assumption_2}
\end{equation}
This assumption is a technical one required to allow solutions
$u^\rho$ of the scheme to be discontinuous (see
\cref{rem:discontinuous_solutions} for an explanation).
It is readily verified that all schemes in this work satisfy this assumption.


We will show that if a scheme $S$ is monotone, stable, and
consistent, it converges to the unique solution of \cref{eq:pde},
provided that \cref{eq:pde} satisfies a
\cref{eq:nonlocal_definition}-comparison principle.
While our notions of monotonicity and stability are identical to
those in the BSF, our notion of consistency will differ from
\cref{eq:local_consistency}.
We now state precisely these notions.

A scheme $S$ is \emph{stable} if (cf. \cite[Eq. (2.3)]{MR1115933})
\[
  \text{there exists a solution } u^{\rho} \text{ of } \cref{eq:scheme}
  \text{ for each } \rho > 0
  \text{ and } \textstyle{\sup_{\rho>0}} \Vert u^{\rho} \Vert_{\infty} < \infty.
\]
A scheme $S$ is \emph{monotone} if (cf. \cite[Eq. (2.2)]{MR1115933})
\[
  S(\rho, x, r, u, \ell)
  \leq
  S(\rho, x, r, w, \ell)
  \text{ whenever }
  u \geq w \text{ pointwise}.
\]
Note that the monotonicity requirement does not involve the operator
$\mathcal{N}^\rho$ and as such, the reader may be tempted to guess
that high order discretizations of the nonlocal operator
$\mathcal{M}$ are possible.
Unfortunately, this is not the case: high order discretizations are
generally precluded by our notion of consistency (see
\cref{rem:higher_order_interpolation}).

Before we introduce our notion of consistency, we recall half-relaxed limits.
For a family $(z^{\rho})_{\rho>0}$ of real-valued maps from a metric
space such that $(\rho,x)\mapsto z^{\rho}(x)$ is locally bounded, we
define the upper and lower half-relaxed limits $\overline{z}$ and
$\underline{z}$ by
\[
  \overline{z}(x) \coloneqq \limsup_{\substack{
      \rho \downarrow 0 \\
      y \rightarrow x
  }}
  z^{\rho}(y)
  \mathand
  \underline{z}(x) \coloneqq \liminf_{\substack{
      \rho \downarrow 0\\
      y \rightarrow x
  }}
  z^{\rho}(y)
\]
where
\[
  \limsup_{\substack{
      \rho \downarrow 0 \\
      y \rightarrow x
  }}
  z^{\rho}(y)
  \coloneqq
  \lim_{\epsilon \downarrow 0} \left(
    \sup \left \{
      z^{\rho}(y) \colon |x - y| < \epsilon \text{ and } 0 < \rho < \epsilon
    \right \}
  \right)
\]
and the limit inferior is defined similarly (cf. \cite[Definition
1.4]{MR1484411}).
Now, a scheme $S$ is \emph{nonlocally consistent} if for each family
$(w^{\rho})_{\rho>0}$
of uniformly bounded real-valued maps from $\overline{\Omega}$,
$\varphi\in C^{2}(\overline{\Omega})$, and $x\in\overline{\Omega}$, we have
\begin{subequations}
  \begin{align}
    \limsup_{\substack{
        \rho \downarrow 0 \\
        y \rightarrow x \\
        \xi \rightarrow 0
    }}
    S(\rho, y, \varphi(y) + \xi, [\varphi + \xi]_y, \mathcal{N}^\rho w^\rho(y))
    & \leq
    F^*(x, \varphi(x), D \varphi(x), D^2 \varphi(x), \mathcal{M}
    \underline{w}(x))
    \label{eq:superconsistency}
    \\
    \liminf_{\substack{
        \rho \downarrow 0 \\
        y \rightarrow x \\
        \xi \rightarrow 0
    }}
    S(\rho, y, \varphi(y) + \xi, [\varphi + \xi]_y, \mathcal{N}^\rho w^\rho(y))
    & \geq
    F_*(x, \varphi(x), D \varphi(x), D^2 \varphi(x), \mathcal{M}
    \overline{w}(x)).
    \label{eq:subconsistency}
  \end{align}
  \label{eq:consistency}%
\end{subequations}
While \cref{eq:local_consistency} and \cref{eq:consistency} are
aesthetically similar, the latter does not apply test functions to
the operator $\mathcal{M}$.
This is perhaps not too surprising if we think about
\cref{eq:local_consistency} and \cref{eq:consistency} as notions of
consistency relating to \cref{eq:local_definition} and
\cref{eq:nonlocal_definition}, respectively: in
\cref{def:viscosity_solution_2}, test functions \emph{are} applied to
the operator $\mathcal{M}$ while in \cref{def:viscosity_solution},
test functions \emph{are not} applied to $\mathcal{M}$.
We are now ready to state the BSF-type convergence result.
\begin{theorem}
  \label{thm:convergence_result}
  Let $S$ be a monotone, stable, and nonlocally consistent scheme
  whose limiting equation \cref{eq:pde} satisfies a
  \cref{eq:nonlocal_definition}-comparison principle (in
  $B(\overline{\Omega})$).
  Then, as $\rho \downarrow 0$, the solution $u^{\rho}$ of
  \cref{eq:scheme} converges locally uniformly to the unique
  \cref{eq:nonlocal_definition}-solution of \cref{eq:pde} in
  $B(\overline{\Omega})$.
\end{theorem}

The proof is very similar to that of \cite[Theorem 2.1]{MR1115933}.
Regardless, we give a detailed proof so that the reader can see the
motivation behind nonlocal consistency.

\begin{proof}
  Let $\overline{u} \geq \underline{u}$ denote the half-relaxed
  limits of the family $(u^{\rho})_{\rho>0}$.
  We claim that $\overline{u}$ is a subsolution and $\underline{u}$
  is a supersolution of \cref{eq:pde}.
  In this case, the comparison principle yields $\overline{u} =
  \overline{u}^* \leq \underline{u}_* = \underline{u}$, from which
  the desired result follows.

  We now prove that $\overline{u}$ is a subsolution (that
  $\underline{u}$ is a supersolution is proved similarly).
  To this end, let $x\in\overline{\Omega}$ be a local maximum point
  of $\overline{u} - \varphi$ for some $\varphi \in C^{2}(\overline{\Omega})$.
  By definition, we can find a neighbourhood (relative to
  $\overline{\Omega}$) $U$ whose closure is compact and on which $x$
  is a global maximum point of $\overline{u} - \varphi$.
  Without loss of generality, we may assume the maximum is strict,
  $\overline{u}(x) = \varphi(x)$, and $\varphi \geq 1 + \sup_{\rho}
  \Vert u^{\rho} \Vert_{\infty}$ outside $U$.
  By the definition of $\overline{u}$, we can find a sequence
  $(\rho_n, x_n)_n$ such that $\rho_n \downarrow 0$, $x_n \rightarrow
  x$, and $u^{\rho_n}(x_n) \rightarrow \overline{u}(x)$.
  For each $n$, pick $y_n \in U$ such that
  \begin{equation}
    (u^{\rho_n} - \varphi)(y_n) + h_1(\rho_n)
    \geq \sup_{y \in U} \left\{
      (u^{\rho_n} - \varphi)(y)
    \right\}
    \label{eq:sup}
  \end{equation}
  where $h_1$ is the function in \cref{eq:margin_assumption_1}.
  Now, pick a subsequence of $(\rho_n, x_n, y_n)_n$ such that its
  last argument converges to some point $\hat{y} \in \overline{U}$.
  With a slight abuse of notation, relabel this subsequence $(\rho_n,
  x_n, y_n)_n$.
  It follows that
  \begin{multline*}
    0 = (\overline{u} - \varphi)(x)
    = \lim_{n \rightarrow \infty} \left\{
      (u^{\rho_n} - \varphi)(x_n)
    \right\}
    \leq \limsup_{n \rightarrow \infty} \left\{
      (u^{\rho_n} - \varphi)(y_n) + h_1(\rho_n)
    \right\} \\
    \leq \limsup_{\substack{
        \rho \downarrow 0 \\
        y \rightarrow \hat{y}
    }}
    \left\{
      (u^\rho - \varphi)(y) + h_1(\rho)
    \right\}
    \leq (\overline{u} - \varphi)(\hat{y}).
  \end{multline*}
  Because $x$ was assumed to be a strict maximum point, the above
  implies $\hat{y} = x$.
  Letting $\xi_n \coloneqq (u^{\rho_n} - \varphi)(y_n) +
  h_1(\rho_n)$, we have $\xi_n \rightarrow 0$ and $u^{\rho_n} \leq
  \varphi + \xi_n$.

  The definition of $u^{\rho}$ and the monotonicity of $S$ yield
  \begin{multline*}
    0
    = S(\rho_n, y_n, u^{\rho_n}(y_n), [u^{\rho_n}]_{y_n},
    \mathcal{N}^{\rho_n} u^{\rho_n}(y_n))
    \\
    \geq S(\rho_n, y_n, \varphi(y_n) + \xi_n - h_1(\rho_n), [\varphi
    + \xi_n]_{y_n}, \mathcal{N}^{\rho_n} u^{\rho_n}(y_n)).
  \end{multline*}
  Taking limit inferiors, we get
  \[
    0 \geq \liminf_{n \rightarrow \infty}
    S(\rho_n, y_n, \varphi(y_n) + \xi_n - h_1(\rho_n), [\varphi +
    \xi_n]_{y_n}, \mathcal{N}^{\rho_n} u^{\rho_n}(y_n)).
  \]
  Now, employing \cref{eq:margin_assumption_1},
  \cref{eq:margin_assumption_2}, and nonlocal consistency,
  \begin{multline*}
    0
    \geq \liminf_{\substack{
        \rho \downarrow 0 \\
        y \rightarrow x \\
        \xi \rightarrow 0
    }}
    S(\rho, y, \varphi(y)+\xi, [\varphi + \xi]_y, \mathcal{N}^\rho u^\rho(y))
    - h_2(\rho, y) \\
    \geq F_*(x, \varphi(x), D \varphi(x), D^2 \varphi(x), \mathcal{M}
    \overline{u}(x)),
  \end{multline*}
  which is the desired inequality, since $\overline{u}(x) = \varphi(x)$.
\end{proof}

\begin{rem}
  \label{rem:discontinuous_solutions}
  \cref{eq:margin_assumption_1} and \cref{eq:margin_assumption_2}
  were needed in the proof above since the supremum in \cref{eq:sup}
  is not necessarily attained at a point in $\overline{U}$.
  If we restrict our attention to the case in which solutions of the
  scheme $u^\rho$ are continuous, then this problem disappears.
  However, our analysis requires discontinuous solutions (see, e.g.,
  \cref{eq:piecewise_constant} in the sequel).
  This issue is discussed further in the first author's note \cite{253954}.
\end{rem}

\begin{rem}
  \label{rem:wlog}
  Recall that the functions $(w^\rho)_{\rho > 0}$ appearing in the
  nonlocal consistency inequalities \cref{eq:consistency} are an
  arbitrary family of uniformly bounded real-valued maps.
  However, a close inspection of the proof of
  \cref{thm:convergence_result} reveals that we only use these
  inequalities with $w^\rho = u^\rho$ where $u^\rho$ is a solution of
  the scheme.
  Therefore, in establishing nonlocal consistency we can, without
  loss of generality, assume that $w^\rho$ is a solution of the scheme.
\end{rem}

As in \cite{MR1115933}, the result above extends to solutions that
are not necessarily bounded.
In particular, let $r:\mathbb{R}\rightarrow\mathbb{R}$ be a map
defined by $r(x)\coloneqq \const \, (1 + x^d)$ where $d$ is a
positive integer and let $R(\overline{\Omega})$ be the set of maps
$u:\overline{\Omega}\rightarrow\mathbb{R}$ satisfying
\[
  \left|u(x)\right| \leq r(|x|) \text{ for all } x \in \overline{\Omega}.
\]
Now, relax the stability condition to read
\[
  \text{there exists a solution } u^{\rho} \in R(\overline{\Omega})
  \text{ of } \cref{eq:scheme} \text{ for each } \rho>0
\]
and the consistency condition by requiring \cref{eq:consistency} to
hold more generally for families $(w^{\rho})_{\rho>0}\subset
R(\overline{\Omega})$ not necessarily uniformly bounded.
Then, by replacing instances of $B(\overline{\Omega})$ by
$R(\overline{\Omega})$ in \cref{thm:convergence_result}, we obtain a
straightforward relaxation of \cref{thm:convergence_result} which
allows for solutions of polynomial growth.
In the proof, the boundedness argument outside $U$ should be replaced
by a $C^2$ modification of the test function which agrees with
$\varphi$ near $x$ and dominates the growth function $r$ outside $U$.
\cref{thm:convergence_result} can also be extended to the case in
which the PDE involves a family of nonlocal operators (cf. \cite[Eq.
(2)]{MR2422079}):
\[
  F(x, u(x), D u(x), D^2 u(x), \{ \mathcal{M}_\alpha u(x) \}_\alpha) = 0.
\]

\section{Hamilton-Jacobi-Bellman quasi-variational inequality}
\label{sec:hjbqvi}

Using the results of the previous section, we now prove the
convergence of two implicit schemes for HJBQVIs.
Both schemes appear in the first author's work \cite{MR3493959}.
We make the following assumptions about the various quantities
appearing in the HJBQVI \cref{eq:hjbqvi}:

\begin{enumerate}[label=(H\arabic*)]
  \item \label{enu:vanilla_start} $f$, $g$ $\mu$, $\sigma$, $\Gamma$,
    and $K$ are continuous functions with $f$ and $g$ bounded and
    $\mathcal{M} g \leq g$ where
    $
    \mathcal{M} g(x) \coloneqq \sup_{z \in Z(T, x)} \{
      g(x + \Gamma(T, x, z)) + K(T, x, z)
    \}
    $.
  \item \label{enu:vanilla_middle} $B$ is a nonempty compact metric
    space and $Z(t, x) \subset Y$ is a nonempty compact set for each $(t, x)$.
  \item \label{enu:vanilla_end} $\sup_{t, x, z} K(t, x, z)<0$.
\end{enumerate}

It is well-known (see, e.g., \cite{MR2109687}) that the HJBQVI is the
dynamic programming equation associated to the optimal control
problem whose value function is given by
\begin{equation}
  u(t, x) = \sup_\gamma \mathbb{E} \left [
    g(X_T^{t, x, \gamma})
    + \int_t^T f(s, X_s^{t, x, \gamma}, b_s) ds
    + \sum_{t \leq \tau_i \leq T} K(\tau_i, X_{\tau_i-}^{t, x, \gamma}, z_i)
  \right ]
  \label{eq:control_problem}
\end{equation}
where
\[
  X_s^{t, x, \gamma}
  \coloneqq x
  + \int_t^s \mu(X_v^{t, x, \gamma}, b_v) dv
  + \int_t^s \sigma(X_v^{t, x, \gamma}, b_v) dW_v
  + \sum_{t \leq \tau_i \leq s} \Gamma(\tau_i, X_{\tau_i-}^{t, x, \gamma}, z_i).
\]
Here, $W$ is a $d$-dimensional Brownian motion with filtration
$(\mathcal{F}_t)_{t \geq 0}$, $\gamma = (a, b)$ is a combined
stochastic and impulse control\footnote{The interested reader can
  review
  \cite{MR673169,MR1240006,MR1306930,MR2109687,MR2568293,MR2735526,MR2812853,MR3053571,MR3070528,MR3071398,MR3145856,MR3200009,MR3615458,MR3626187,MR3654873}
  for information and recent developments on impulse control. We do
not claim this list to be exhaustive.} with $a = (\tau_i, z_i)_{i
\geq 1}$, $(\tau_i)_{i \geq 1}$ being an increasing sequence of
stopping times, each $z_i$ being $\mathcal{F}_{\tau_i}$-measurable
and taking values in some compact $Z(\tau_i, X_{\tau_i-}^{t, x,
\gamma})$, and $b$ being a progressively measurable process taking
values in some compact $B$.

In the context of the optimal control problem
\cref{eq:control_problem}, the assumption $\mathcal{M} g \leq g$
implies that it is suboptimal to perform an impulse at the terminal time $T$.
The assumption \ref{enu:vanilla_end}, needed to ensure stability, can
be interpreted as the controller paying a cost for the right to
perform an impulse.
Lastly, we mention that the choice to write the functions $\mu$ and
$\sigma$ as independent of the time coordinate $t$ was made only to
simplify notation, and not due to a shortcoming of the theory.

\subsection{Penalty scheme}
\label{subsec:penalty_scheme}

In this subsection, we prove the convergence of a penalty scheme from
\cite[Section 5.2]{MR3493959} to the solution of the HJBQVI.
We start with the one-dimensional case ($d = 1$), deferring for now a
discussion of the complications that arise in higher dimensions.
In this case, the operator $L_b$ simplifies to
\[
  L_b u(t, x) = \mu(x, b) u_x(t, x) + \frac{1}{2} \sigma(x, b)^2 u_{xx}(t, x).
\]

To simplify presentation, we assume uniformly spaced grid points $\{
n \Delta t \}_{n = 0}^N$ and $\{ j \Delta x \}_{j = -M}^M$ in time
and space with $T = N \Delta t$ and $M > 0$.
For brevity, we let $x_j \coloneqq j \Delta x$.
We assume $\Delta t$ and $\Delta x$ are related through a positive
parameter $\rho$ as per
\[
  \Delta t = \const \rho
  \mathand
  \Delta x = \const \rho.
\]
For PDEs on unbounded domains, the usual approach is to truncate the
domain to a bounded region (e.g., $[-Q, Q]$) and impose artificial
boundary conditions.
Since these arguments are somewhat standard, we will assume
\begin{equation}
  Q \coloneqq M \Delta x \rightarrow \infty
  \mathas
  \rho \downarrow 0
  \label{eq:infinite_grid}
\end{equation}
in order to ignore artificial boundary conditions in our convergence
analysis and focus instead on the main difficulties specific to the HJBQVI.
For example, to satisfy \cref{eq:infinite_grid}, we may take $M =
\const \rho^{-1-\alpha}$ for some $\alpha > 0$ so that $Q = \const
\rho^{-\alpha}$.
Ignoring artificial boundary conditions to focus on the main points
is a common practice in the literature on elliptic equations (see,
  e.g., \cite[Section 3]{MR3042570}, in which the authors take an
  infinite grid to numerically solve a Hamilton-Jacobi-Bellman-Isaacs
(HJBI) equation).

Writing $u_j^n$ for the quantity $u(n \Delta t, x_j)$, let $u^n
\coloneqq(u_j^n)_{j = -M}^M$ be the image of $u(n \Delta t, \cdot)$
on the spatial grid.
Note that $u^n$ is a vector in $\mathbb{R}^{2M + 1}$.
In order to present the penalty scheme, our first task is to define
approximations $\mathcal{D}$ and $\mathcal{D}^2$ of the first and
second derivatives $u_x$ and $u_{xx}$ in the sense that
\[
  u_x(n \Delta t, x_j) \approx (\mathcal{D} u^n)_j
  \mathand
  u_{xx}(n \Delta t, x_j) \approx (\mathcal{D}^2 u^n)_j.
\]
For the second derivative, we use a standard three point stencil:
\begin{equation}
  (\mathcal{D}^2 u^n)_j
  \coloneqq \frac{ u_{j - 1}^n - 2 u_j^n + u_{j + 1}^n }{ (\Delta x)^2 }.
  \label{eq:stencil_2}
\end{equation}
For the first derivative, we use an upwind stencil:
\begin{equation}
  (\mathcal{D} u^n)_j
  \coloneqq
  \frac{1}{\Delta x}
  \begin{cases}
    u_{j + 1}^n - u_j^n
    & \text{if } \mu(x_j, b) \geq 0
    \\
    u_j^n - u_{j - 1}^n
    & \text{if } \mu(x_j, b) < 0.
  \end{cases}
  \label{eq:stencil_1}
\end{equation}
These stencils should be modified accordingly in the case of a
non-uniform spatial grid.
Since the expressions \cref{eq:stencil_2} and \cref{eq:stencil_1} are
not well-defined at $j = \pm M$, we set $\mathcal{D}$ and
$\mathcal{D}^2$ to zero there for convenience:
\[
  (\mathcal{D} u^n)_{\pm M} \coloneqq 0
  \mathand
  (\mathcal{D}^2 u^n)_{\pm M} \coloneqq 0.
\]
Noting that the upwind direction in \cref{eq:stencil_1} depends on
the choice of $b$, we will sometimes write $\mathcal{D}_b$ in lieu of
$\mathcal{D}$ to make this dependence explicit.
\begin{rem}
  \label{rem:truncated_domain}
  Assumption \cref{eq:infinite_grid} is computationally expensive to
  implement in practice and as such, a practical implementation
  should instead take $Q$ to be constant independent of $\rho$.
  In this case, we can interpret
  \[
    u_x(n \Delta t, \pm Q) \approx (\mathcal{D} u^n)_{\pm M} = 0
    \mathand
    u_{xx}(n \Delta t, \pm Q) \approx (\mathcal{D}^2 u^n)_{\pm M} = 0
  \]
  as imposing the Neumann boundary condition $u_x(t, \pm Q) =
  u_{xx}(t, \pm Q) = 0$
  (other boundary conditions can be imposed by modifying the scheme
  appropriately).
  In the context of the optimal control problem
  \cref{eq:control_problem}, this corresponds to (formally) modifying
  the drift $\mu$ and diffusion $\sigma$ of $X^{t, x, \gamma}$ by
  setting them to zero outside of $(-Q, Q)$.
\end{rem}

Since the control set $B$ appearing in \cref{eq:hjbqvi} may be
infinite, we replace it with a nonempty finite subset $B^\rho \subset B$.
Similarly, since the set $Z(t, x)$ appearing in the intervention
operator \cref{eq:intervention} may be infinite at each point $(t,
x)$, we replace it with a nonempty finite subset $Z^\rho(t, x) \subset Z(t, x)$.
To ensure consistency, we require:
\begin{enumerate}[label=(H\arabic*),start=4]
  \item
    \label{enu:comparison_end}
    As $\rho \downarrow 0$, $B^\rho$ converges to $B$ (in the
    Hausdorff metric) and $Z^\rho$ converges locally uniformly to $Z$
    (in the Hausdorff metric).
    Moreover, $(t, x) \mapsto Z^\rho(t, x)$ is continuous (in the
    Hausdorff metric) for each $\rho$.
\end{enumerate}
Unless otherwise mentioned, we will always assume the Hausdorff
metric when discussing compact sets.
Note that the above implies that $Z$ is also continuous.

Next, note that the point $x + \Gamma(t, x, z)$ appearing in the
intervention operator \cref{eq:intervention} is not necessarily a
point on the grid.
Therefore, a discretization of the intervention operator requires interpolation.
We use $\interp(u^n, x)$ to denote the value of $u(n \Delta t, x)$ as
approximated by a standard monotone linear interpolant.
Precisely, for a point $x$ within the grid (i.e., $x_{-M} < x < x_M$), we define
\begin{equation}
  \interp(u^n,x) \coloneqq \alpha u_{k + 1}^n + \left(1 - \alpha \right) u_k^n
  \label{eq:interpolation}
\end{equation}
where $k$ is the unique integer satisfying $x_k \leq x < x_{k + 1}$
and $\alpha = (x - x_k) / (x_{k+1} - x_k)$.
If $x$ is not within the grid (i.e., $x \leq x_{-M}$ or $x \geq x_M$), we define
\[
  \interp(u^n,x) \coloneqq
  \begin{cases}
    u_{-M}^n & \text{if } x \leq x_{-M} \\
    u_M^n & \text{if } x \geq x_M.
  \end{cases}
\]
We can now discretize the intervention operator according to
\[
  (\mathcal{M}^\rho u^n)_j
  \coloneqq \max_{z \in Z^\rho(n \Delta t, x_j)} \left \{
    \interp(u^n, x_j + \Gamma(n \Delta t, x_j, z))
    + K(n \Delta t, x_j, z)
  \right \}.
\]
Using the notation above, $\mathcal{M} u(n \Delta t, x_j) \approx
(\mathcal{M}^\rho u^n)_j$.

In \cite{MR673169}, the authors show that a solution of the HJBQVI
\cref{eq:hjbqvi} can be constructed as the limit of $u^\epsilon$ as
$\epsilon \downarrow 0$ where $u^\epsilon$ solves the so-called
``penalized'' problem
\begin{equation}
  \begin{cases}
    \displaystyle{
      - \sup_{b \in B} \left \{
        \frac{\partial u}{\partial t}
        + L_b u
        + f(\cdot, b)
      \right \}
      - \left( \frac{\mathcal{M} u - u}{\epsilon} \right)^+
      = 0,
    }
    & \text{on } [0, T) \times \mathbb{R}^d
    \\
    u(T, \cdot) = g,
    & \text{on } \mathbb{R}^d
  \end{cases}
  \label{eq:penalized}
\end{equation}
where $(a)^+ = \max (a, 0)$.
The penalty scheme is just a discretization of \cref{eq:penalized}:
\begin{equation}
  \begin{cases}
    \displaystyle{
      - \max_{b \in B^\rho} \left \{
        \frac{u_j^{n + 1} - u_j^n}{\Delta t}
        + (L_b^\rho u^n)_j
        + f_j^n(b)
      \right \}
      - \left( \frac{ (\mathcal{M}^\rho u^n)_j - u_j^n }{\epsilon} \right)^+
      = 0,
    }
    & \text{for } n < N
    \\
    u_j^N = g(x_j)
  \end{cases}
  \label{eq:penalty}
\end{equation}
where (recall that $d = 1$)
\[
  (L_b^\rho u^n)_j
  \coloneqq \mu_j(b) (\mathcal{D}_b u^n)_j
  + \frac{1}{2} \sigma_j(b)^2 (\mathcal{D}^2 u^n)_j
\]
and we have introduced the shorthand $f_j^n(b) \coloneqq f(n \Delta
t, x_j, b)$, $\mu_j(b) \coloneqq \mu(x_j, b)$, and $\sigma_j(b)
\coloneqq \sigma(x_j, b)$.
To ensure consistency, we will always assume
\[
  \epsilon \downarrow 0
  \mathas
  \rho \downarrow 0.
\]

\begin{rem}
  \label{rem:iterative_method}
  \cref{eq:penalty} defines $(2M+1)(N+1)$ ``discrete'' equations.
  In \cite[Section 5.2]{MR3493959}, it is proved that these equations
  admit a unique solution.
  Morever, the solution at the $n$-th timestep $u^n$ can be computed
  by policy iteration or by value iteration.
  This is the motivation for discretizing \cref{eq:penalized} instead
  of the HJBQVI \cref{eq:hjbqvi} directly since in the latter case,
  the resulting discrete equations may not admit any solutions, and
  even when they do, their computation may be nontrivial
  \cite[Section 5.1]{MR3493959}.
\end{rem}

Note that the unique solution of the discrete equations
\cref{eq:penalty} is defined only at grid points $(n \Delta t, x_j)$.
Following \cite[Pg. 281]{MR1115933}, we extend it to all points in
$[0, T] \times \mathbb{R}$ by
\begin{equation}
  u^\rho(t, x)
  \coloneqq \sum_{\substack{
      0 \leq n \leq N \\
      -M \leq j \leq M
  }} u_j^n \cdot \boldsymbol{1}_{E(n,j)}(t, x)
  \label{eq:piecewise_constant}
\end{equation}
where $E(n, j) \coloneqq [(n - \nicefrac{1}{2}) \Delta t, (n +
\nicefrac{1}{2}) \Delta t) \times [(j - \nicefrac{1}{2}) \Delta x, (j
+ \nicefrac{1}{2}) \Delta x)$ and $\boldsymbol{1}_G$ is the indicator
function of the set $G$.
Note that the extension $u^\rho$ is a (discontinuous) simple function
(recall \cref{rem:discontinuous_solutions}).
We can now state the convergence result:

\begin{theorem}
  \label{thm:penalty_convergence}
  Suppose \ref{enu:vanilla_start}\textendash \ref{enu:comparison_end}
  and that $\mu$ and $\sigma$ are Lipschitz in $x$ uniformly in $b$
  so that the HJBQVI \cref{eq:hjbqvi} satisfies a
  \cref{eq:nonlocal_definition}-comparison principle.
  Suppose also that \cref{eq:infinite_grid} holds.
  Then, as $\rho \downarrow 0$, the solution $u^\rho$ of the penalty
  scheme converges locally uniformly to the unique bounded solution
  of the HJBQVI.
\end{theorem}

For the reader's convenience, the comparison principle alluded to
above is stated in \cref{app:comparison_principle}.
\ifx\myarticle\undefined
We will spend the remainder of this subsection proving the above by
appealing to \cref{thm:convergence_result} and establishing the
nonlocal consistency of the penalty scheme (the stability and
  monotonicity arguments, which appear in the preprint version
  \href{https://arxiv.org/abs/1705.02922}{arXiv:1705.02922}, are
standard and hence omitted).
\else
We will spend the remainder of this subsection proving the above by
appealing to \cref{thm:convergence_result} and establishing the
monotonicity, stability, and nonlocal consistency of the penalty scheme.
\fi

First, note that the discrete equations \cref{eq:penalty} are
equivalent to, by some algebra,
\[
  \begin{cases}
    \min \left (
      - (\mathcal{L}^\rho u^n)_j,
      u_j^n - (\mathcal{M}^\rho u^n)_j
      - \epsilon (\mathcal{L}^\rho u^n)_j
    \right ) = 0,
    & \text{for } n < N
    \\
    u_j^N = g(x_j)
  \end{cases}
\]
where
\[
  (\mathcal{L}^\rho u^n)_j
  \coloneqq
  \max_{b \in B^\rho} \left \{
    \frac{u_j^{n + 1} - u_j^n}{\Delta t}
    + (L_b^\rho u^n)_j
    + f_j^n(b)
  \right \}.
\]
Next, we write the penalty scheme in the form \cref{eq:scheme} by taking
\begin{equation}
  S( \rho, (n \Delta t, x_j), u_j^n, [u]_j^n, \ell)
  \coloneqq
  \begin{cases}
    \min \left (
      - (\mathcal{L}^\rho u^n)_j,
      u_j^n - \ell
      - \epsilon (\mathcal{L}^\rho u^n)_j
    \right ),
    & \text{if } n < N
    \\
    u_j^n - g(x_j),
    & \text{if } n = N
  \end{cases}
  \label{eq:penalty_scheme_1}
\end{equation}
where $[u]_j^n \coloneqq [u]_{(n \Delta t, x_j)}$ and
\begin{equation}
  \mathcal{N}^\rho u(n \Delta t, x_j)
  \coloneqq (\mathcal{M}^\rho u^n)_j.
  \label{eq:penalty_scheme_2}
\end{equation}
Whenever a scheme is written only at grid points $(n \Delta t, x_j)$,
it is understood that for $(s, y) \in E(n, j)$, the formula for $S$
and $\mathcal{N}^\rho$ is copied from the grid point $(n \Delta t,
x_j)$, with all coefficients, costs, intervention sets, and indices
kept fixed, and with only the finite-difference and interpolation
stencils translated to be based at $(s, y)$.

\ifx\myarticle\undefined
\else

\subsubsection{Monotonicity}
\label{subsubsec:penalty_monotonicity}

Due to our choice of upwind discretization, the penalty scheme is
monotone by construction. We establish this below.

\begin{proof}[Monotonicity proof]
  Let $(n \Delta t, x_j)$ be an arbitrary grid point and let $u$ and
  $w$ be functions in $B([0,T] \times \mathbb{R})$ satisfying $u \geq
  w$ and $u_j^n = w_j^n$ where we have, as usual, employed the
  shorthand $u_j^n = u(n \Delta t, x_j)$ and $w_j^n = w(n \Delta t, x_j)$.
  If $n = N$, then
  \[
    S(\rho, (n \Delta t, x_j), u_j^n, [u]_j^n, \ell) - S(\rho, (n
    \Delta t, x_j), w_j^n, [w]_j^n, \ell)
    = \left( u_j^n - g(x_j) \right) - \left( w_j^n - g(x_j) \right) = 0.
  \]
  Otherwise, defining $v \coloneqq w - u$, we have
  \[
    q \coloneqq \max_{b \in B^\rho} \left \{
      \frac{v_j^{n + 1}}{\Delta t}
      + (L_b^\rho v^n)_j
    \right \}
    \leq 0
  \]
  since the terms $v_j^{n + 1}$ and $(L_b^\rho v^n)_j$ are nonpositive.
  Moreover,
  \begin{equation}
    S(\rho, (n \Delta t, x_j), u_j^n, [u]_j^n, \ell)
    - S(\rho, (n \Delta t, x_j), w_j^n, [w]_j^n, \ell)
    \leq \max \{q, \epsilon q\}
    \leq 0. \qedhere
    \label{eq:penalty_monotonicity}
  \end{equation}
\end{proof}

\subsubsection{Stability}
\label{subsubsec:penalty_stability}

We will prove that a solution of the penalty scheme is bounded by
$\Vert g \Vert_\infty + \Vert f \Vert_\infty T$.
In the context of the optimal control problem
\cref{eq:control_problem}, this bound has the natural interpretation
of being the sum of the maximum possible continuous reward and the
maximum possible reward obtained at the final time:
\[
  g(X_T^{t, x, \gamma}) + \int_t^T f(s, X_s^{t, x, \gamma}, b_s) ds
  \leq \Vert g \Vert_\infty + \Vert f \Vert_\infty T.
\]

\begin{proof}[Stability proof]
  Let $u$ be a solution of the penalty scheme.
  Fix a particular timestep $n < N$ and let $j-$ be such that
  $u_{j-}^n = \min_j u_j^n$.
  It is easily verified that $(L_b^\rho u^n)_{j-} \geq 0$.
  Therefore,
  \begin{multline}
    0
    = - \max_{b \in B^\rho} \left \{
      \frac{u_{j-}^{n + 1} - u_{j-}^n}{\Delta t}
      + (L_b^\rho u^n)_{j-}
      + f_{j-}^n(b)
    \right \}
    - \left( \frac{ (\mathcal{M}^\rho u^n)_{j-} - u_{j-}^n }{\epsilon} \right)^+
    \\
    \leq -\frac{u_{j-}^{n + 1} - u_{j-}^n}{\Delta t} + \Vert f \Vert_\infty.
    \label{eq:penalty_lower}
  \end{multline}
  Now, let $j+$ be such that $u_{j+}^n = \max_j u_j^n$.
  It is easily verified that $(L_b^\rho u^n)_{j+} \leq 0$ and
  $(\mathcal{M}^\rho u^n)_{j+} \leq u_{j+}^n$ (since $K \leq 0$ by
  \ref{enu:vanilla_end}).
  Therefore,
  \begin{multline}
    0
    = - \max_{b \in B^\rho} \left \{
      \frac{u_{j+}^{n + 1} - u_{j+}^n}{\Delta t}
      + (L_b^\rho u^n)_{j+}
      + f_{j+}^n(b)
    \right \}
    - \left( \frac{ (\mathcal{M}^\rho u^n)_{j+} - u_{j+}^n }{\epsilon} \right)^+
    \\
    \geq -\frac{u_{j+}^{n + 1} - u_{j+}^n}{\Delta t} - \Vert f \Vert_\infty.
    \label{eq:penalty_upper}
  \end{multline}
  By \cref{eq:penalty_lower} and \cref{eq:penalty_upper},
  \begin{multline*}
    \min_j u_j^{n + 1} - \Vert f \Vert_\infty \Delta t
    \leq u_{j-}^{n + 1} - \Vert f \Vert_\infty \Delta t
    \leq u_{j-}^n \\
    \leq u_{j+}^n
    \leq u_{j+}^{n + 1} + \Vert f \Vert_\infty \Delta t
    \leq \max_j u_j^{n + 1} + \Vert f \Vert_\infty \Delta t
  \end{multline*}
  and hence $\Vert u^n \Vert_\infty \leq \Vert u^{n + 1} \Vert_\infty
  + \Vert f \Vert_\infty \Delta t$.
  By induction on $n$,
  \[
    \Vert u^n \Vert_\infty
    \leq \Vert u^N \Vert_\infty + \Vert f \Vert_\infty \left( T - n
    \Delta t \right),
  \]
  and the desired result follows since $\Vert u^N \Vert_\infty \leq
  \Vert g \Vert_\infty$.
\end{proof}

\subsubsection{Nonlocal consistency}

\fi

To establish nonlocal consistency, we require three lemmas.
\ifx\myarticle\undefined
The proofs of the first two, which appear in the preprint version
\href{https://arxiv.org/abs/1705.02922}{arXiv:1705.02922}, are
standard and hence omitted.
\else
\fi
The first lemma is a standard result from analysis:

\begin{lemma}
  \label{lem:commute}
  Let $(a_n)_n$, $(b_n)_n$, and $(c_n)_n$ be real sequences with $c_n
  \geq \min ( a_n, b_n )$ (resp. $c_n \leq \min ( a_n, b_n )$) for each $n$.
  Then,
  \begin{align*}
    \liminf_{n \rightarrow \infty} c_n
    & \geq \min (
      \liminf_{n \rightarrow \infty} a_n,
      \liminf_{n \rightarrow \infty} b_n
    ) \\
    \text{(resp. } \limsup_{n \rightarrow \infty} c_n
      & \leq \min (
        \limsup_{n \rightarrow \infty} a_n,
        \limsup_{n \rightarrow \infty} b_n
      )
    \text{)}.
  \end{align*}
\end{lemma}

\ifx\myarticle\undefined
\else

\begin{proof}
  If $c_n \geq \min ( a_n, b_n )$ for each $n$, then
  \[
    \inf_{m \geq n} c_m
    \geq \inf_{m \geq n} \min ( a_m, b_m )
    = \min ( \inf_{m \geq n} a_m, \inf_{m \geq n} b_m ).
  \]
  Taking limits in the above and using the continuity of $(x, y)
  \mapsto \min ( x,y )$ establishes the result.
  The case of $c_n \leq \min ( a_n, b_n )$ is handled similarly.
\end{proof}

\fi

The second lemma justifies approximating the term $\sup_{b \in B} \{
\cdot \}$ in the HJBQVI \cref{eq:hjbqvi} by $\max_{b \in B^\rho} \{ \cdot \}$.
Recall that, as per \ref{enu:comparison_end}, $B^\rho$ converges to
$B$ in the Hausdorff metric.

\begin{lemma}
  \label{lem:cluster_point}
  Let $Y$ be a compact metric space and $\zeta : Y \times B
  \rightarrow \mathbb{R}$ be a continuous function.
  Let $(\rho_m)_m$ be a sequence of positive numbers converging to
  zero and $(y_m)_m$ be a sequence taking values in $Y$ and
  converging to some $\hat{y}$ in $Y$.
  Then,
  \[
    \max_{b \in B^{\rho_m}}
    \zeta(y_m, b) \rightarrow \max_{b \in B} \zeta(\hat{y}, b).
  \]
\end{lemma}

\ifx\myarticle\undefined
\else

\begin{proof}
  Let $\zeta_m \coloneqq \max_{b \in B^{\rho_m}} \zeta(y_m, b)$.
  Since the sequence $(\zeta_m)_m$ is bounded by $\Vert \zeta
  \Vert_{\infty} < \infty$ ($Y\times B$ is compact and $\zeta$ is
  continuous), it is sufficient to show that every convergent
  subsequence of $(\zeta_m)_m$ converges to $\max_{b \in B} \zeta(\hat{y}, b)$.

  Therefore, consider an arbitrary convergent subsequence of
  $(\zeta_m)_m$, and relabel it, with a slight abuse of notation, $(\zeta_m)_m$.
  Choose $\hat{b}$ such that $\zeta(\hat{y}, \hat{b}) = \max_{b \in
  B} \zeta(\hat{y}, b)$.
  Since $B^{\rho_m} \rightarrow B$, we can find a sequence $(b_m)_m$
  such that $b_m \in B^{\rho_m}$ for each $m$ and $b_m \rightarrow
  \hat{b} \in B$.
  Therefore,
  \[
    \lim_{m \rightarrow \infty} \zeta_m
    = \lim_{m \rightarrow \infty} \max_{b \in B^{\rho_m}} \zeta(y_m, b)
    \geq \lim_{m \rightarrow \infty} \zeta(y_m, b_m)
    = \max_{b \in B} \zeta(\hat{y}, b).
  \]
  Furthermore, note that
  \[
    \lim_{m \rightarrow \infty} \zeta_m
    = \lim_{m \rightarrow \infty} \max_{b \in B^{\rho_m}} \zeta(y_m, b)
    \leq \lim_{m \rightarrow \infty} \max_{b \in B} \zeta(y_m, b)
    = \max_{b \in B} \zeta(\hat{y}, b)
  \]
  where we have obtained the last inequality by the continuity of $y
  \mapsto \max_{b \in B} \zeta(y, b)$.

  To see that $y \mapsto \max_{b \in B} \zeta(y, b)$ is continuous,
  let $y_0 \in Y$ and $\epsilon > 0$ be arbitrary.
  Since $Y \times B$ is compact, $\zeta$ is uniformly continuous.
  Therefore, we can pick $\delta > 0$ such that for all $b \in B$ and
  $y_1 \in Y$ with $|y_0 - y_1| < \delta$, $| \zeta(y_0, b) -
  \zeta(y_1, b) | < \epsilon$ and hence
  \[
    \left| \max_{b \in B} \zeta(y_0, b) - \max_{b \in B} \zeta(y_1, b) \right|
    \leq \max_{b \in B} \left| \zeta(y_0, b) - \zeta(y_1, b) \right|
    < \epsilon. \qedhere
  \]
\end{proof}

\fi

The last lemma characterizes $\mathcal{M}^\rho$ as a legitimate
approximation of $\mathcal{M}$ as $\rho \downarrow 0$.
To state it, we first need to introduce some notation.
Recall that $\Delta t$ and $\Delta x$ are functions of the parameter
$\rho$ (we write $\Delta t^\rho$ and $\Delta x^\rho$ to make this
dependence explicit).
Now, for a function $(\rho,n,j)\mapsto h_j^n(\rho)$ of the parameter
$\rho$ and grid indices $n$ and $j$, we define
\begin{multline*}
  \limsup_{\substack{
      \rho\downarrow0\\
      (n\Delta t,j\Delta x)\rightarrow(t,x)
  }}
  h_{j}^{n}(\rho)
  \coloneqq \sup \biggl \{
    \lim_{m\rightarrow\infty}h_{j_{m}}^{n_{m}}(\rho_{m}) \colon
    (\rho_m, n_m, j_m)_m \text{ is a sequence}
    \\
    \text{such that } (h_{j_{m}}^{n_{m}}(\rho_{m}))_m \text{ converges, }
    \rho_{m}\downarrow0 \text{, and }
    (n_{m}\Delta t^{\rho_{m}},j_{m}\Delta x^{\rho_{m}})\rightarrow(t,x)
  \biggr \}.
\end{multline*}
The limit inferior is defined by $\liminf h_j^n(\rho) \coloneqq
-\limsup (-h_j^n(\rho))$.

\begin{lemma}
  \label{lem:intervention_limits}
  Let $(w^\rho)_{\rho > 0}$ be a family of uniformly bounded
  real-valued maps from $[0,T] \times\mathbb{R}$ with half-relaxed
  limits $\overline{w}$ and $\underline{w}$.
  Denote by $w^{n, \rho} \coloneqq (w^\rho(n \Delta t^\rho, j \Delta
  x^\rho))_{j = -M}^M$ the image of $w^\rho(n \Delta t^\rho, \cdot)$
  on the spatial grid.
  Then, for any $(t, x) \in [0, T] \times \mathbb{R}$,
  \[
    \mathcal{M} \underline{w}(t, x)
    \leq \liminf_{\substack{
        \rho \downarrow 0 \\
        (n \Delta t, j \Delta x) \rightarrow (t, x)
    }} (\mathcal{M}^\rho w^{n, \rho})_j
    \leq \limsup_{\substack{
        \rho \downarrow 0 \\
        (n \Delta t, j \Delta x) \rightarrow (t, x)
    }} (\mathcal{M}^\rho w^{n, \rho})_j
    \leq \mathcal{M} \overline{w}(t, x).
  \]
\end{lemma}

\begin{proof}
  We first prove the leftmost inequality.
  Let $(\rho_m, n_m, j_m)_m$ be an arbitrary sequence satisfying
  $\rho_m \downarrow 0$ and $(n_m \Delta t^{\rho_m}, j_m \Delta
  x^{\rho_m}) \rightarrow (t, x)$.
  Define $s_m \coloneqq n_m \Delta t^{\rho_m}$ and $y_m \coloneqq j_m
  \Delta x^{\rho_m}$ for brevity.
  Now, let $\delta > 0$ and choose $z^\delta \in Z(t, x)$ such that
  \[
    \mathcal{M} \underline{w}(t, x)
    \leq \underline{w}(t, x + \Gamma(t, x, z^{\delta})) + K(t, x,
    z^{\delta}) + \delta.
  \]
  By \ref{enu:comparison_end}, we can pick a sequence $(z_m)_m$ such
  that $z_m \rightarrow z^\delta$ and $z_m \in Z^{\rho_m}(s_m, y_m)$
  for each $m$.
  For brevity, we write $w_j^m$ for the quantity $w^{\rho_m}(s_m, j
  \Delta x^{\rho_m})$.
  Noting that $w_j^m$ is the $j$-th entry of the vector $w^{n_m,
  \rho_m}$ defined in the lemma statement,
  \begin{align*}
    (\mathcal{M}^{\rho_m} w^{n_m, \rho_m})_{j_m}
    & = \max_{z \in Z^{\rho_m}(s_m, y_m)} \left \{
      \interp(
        w^{n_m, \rho_m},
        y_m + \Gamma(s_m, y_m, z)
      )
      {+} K(s_m, y_m, z)
    \right \} \\
    & \geq \interp(
      w^{n_m, \rho_m},
      y_m + \Gamma(s_m, y_m, z_m)
    )
    + K(s_m, y_m, z_m) \\
    & = \alpha_m w_{k_m + 1}^m + (1 - \alpha_m) w_{k_m}^m
    + K(s_m, y_m, z_m)
  \end{align*}
  where $0 \leq \alpha_m \leq 1$ and, for $m$ sufficiently large,
  \begin{equation}
    k_m \Delta x^{\rho_m}
    \leq y_m + \Gamma(s_m, y_m, z_m)
    < \left( k_m + 1 \right) \Delta x^{\rho_m}
    \label{eq:intervention_limits_1}
  \end{equation}
  (recall \cref{eq:interpolation}).
  Therefore, by \cref{lem:commute},
  \begin{align}
    \liminf_{m \rightarrow \infty} (\mathcal{M}^{\rho_m} w^{n_m, \rho_m})_{j_m}
    & \geq \liminf_{m \rightarrow \infty} \left \{
      \min (
        w_{k_m + 1}^m,
        w_{k_m}^m
      )
      + K(s_m, y_m, z_m)
    \right \} \nonumber \\
    & \geq \min \left (
      \liminf_{m \rightarrow \infty} w_{k_m + 1}^m,
      \liminf_{m \rightarrow \infty} w_{k_m}^m
    \right )
    + K(t, x, z^\delta).
    \label{eq:intervention_limits_2}
  \end{align}
  Now, by \cref{eq:intervention_limits_1}, $(s_m, k_m \Delta
  x^{\rho_m}) \rightarrow (t, x + \Gamma(t, x, z^\delta))$ as $m
  \rightarrow \infty$.
  Therefore, by the definition of the half-relaxed limit $\underline{w}$,
  \begin{equation}
    \liminf_{m \rightarrow \infty} w_{k_m}^m
    = \liminf_{m \rightarrow \infty} w^{\rho_m}(s_m, k_m \Delta x^{\rho_m})
    \geq \underline{w}(t, x + \Gamma(t, x, z^\delta))
    \label{eq:intervention_limits_3}
  \end{equation}
  and
  \begin{equation}
    \liminf_{m \rightarrow \infty} w_{k_m + 1}^m
    \geq \underline{w}(t, x + \Gamma(t, x, z^\delta)).
    \label{eq:intervention_limits_4}
  \end{equation}
  Applying \cref{eq:intervention_limits_3} and
  \cref{eq:intervention_limits_4} to \cref{eq:intervention_limits_2},
  \begin{equation}
    \liminf_{m \rightarrow \infty}
    (\mathcal{M}^{\rho_m} w^{n_m, \rho_m})_{j_m}
    \geq \underline{w}(t, x + \Gamma(t, x, z^\delta)) + K(t, x, z^\delta)
    \geq \mathcal{M} \underline{w}(t, x) - \delta.
    \label{eq:intervention_limits_5}
  \end{equation}
  Since $\delta$ is arbitrary, the desired result follows.

  We now handle the rightmost inequality.
  Let $(\rho_m, n_m, j_m)_m$ be a sequence as in the previous
  paragraph with $s_m$ and $y_m$ defined as before.
  Since $Z^{\rho_m}(s_m, y_m)$ is by definition a finite set (and
  hence compact), for each $m$, there exists a $z_m \in
  Z^{\rho_m}(s_m, y_m)$ such that
  \begin{align*}
    (\mathcal{M}^{\rho_m} w^{n_m, \rho_m})_{j_m}
    & = \max_{z \in Z^{\rho_m}(s_m, y_m)} \left \{
      \interp(
        w^{n_m, \rho_m},
        y_m + \Gamma(s_m, y_m, z)
      )
      {+} K(s_m, y_m, z)
    \right \} \\
    & = \interp(
      w^{n_m, \rho_m},
      y_m + \Gamma(s_m, y_m, z_m)
    )
    + K(s_m, y_m, z_m) \\
    & = \alpha_m w_{k_m + 1}^m + (1 - \alpha_m) w_{k_m}^m + K(s_m, y_m, z_m)
  \end{align*}
  where $0 \leq \alpha_m \leq1$ and $k_m$ satisfies
  \cref{eq:intervention_limits_1}.
  Now, consider a subsequence of $(z_m)_m$ along which the limit superior
  \[
    \limsup_{m \rightarrow \infty} (\mathcal{M}^{\rho_m} w^{n_m, \rho_m})_{j_m}
  \]
  is attained.
  With a slight abuse of notation, relabel this subsequence $(z_m)_m$.
  By \ref{enu:comparison_end}, $(z_m)_m$ is contained in a compact
  set and hence we can assume this subsequence was chosen to be
  convergent with limit $\hat{z}$.
  By \ref{enu:comparison_end}, the right hand side of the inequality
  \[
    d(\hat{z}, Z(t, x)) \leq d(\hat{z}, z_m) + d(z_m, Z(t, x))
  \]
  approaches zero as $m \rightarrow \infty$ and hence we can conclude
  $\hat{z} \in Z(t, x)$.
  Therefore, similarly to how we established
  \cref{eq:intervention_limits_2} and \cref{eq:intervention_limits_5}
  in the previous paragraph,
  \begin{align*}
    \lim_{m \rightarrow \infty} (\mathcal{M}^{\rho_m} w^{n_m, \rho_m})_{j_m}
    & \leq \limsup_{m \rightarrow \infty} \{
      \max (
        w_{k_m + 1}^m,
        w_{k_m}^m
      )
      + K(s_m, y_m, z_m)
    \} \nonumber \\
    & \leq \max \left (
      \limsup_{m \rightarrow \infty} w_{k_m + 1}^m,
      \limsup_{m \rightarrow \infty} w_{k_m}^m
    \right )
    + K(t, x, \hat{z}) \\
    & \leq \overline{w}(t, x + \Gamma(t, x, \hat{z})) + K(t, x, \hat{z}) \\
    & \leq \mathcal{M} \overline{w}(t, x). \qedhere
  \end{align*}
\end{proof}

\begin{rem}
  \label{rem:higher_order_interpolation}
  The proof above relies heavily on the fact that $\interp$ is a
  linear interpolant and hence the ``interpolation weights'' $\alpha$
  and $(1 - \alpha)$ in \cref{eq:interpolation} are nonnegative.
  A quadratic interpolant, for example, takes the form
  \[
    \operatorname{quad-interp}(u^n, x)
    = \alpha u_{k - 1}^n + \left( 1 - \alpha - \beta \right) u_k^n +
    \beta u_{k + 1}^n
  \]
  where it is not necessarily the case that the coefficients
  $\alpha$, $(1 - \alpha - \beta)$, and $\beta$ are nonnegative.
  This suggests that higher order discretizations of $\mathcal{M}$
  are generally precluded by the nonlocal consistency requirement.
\end{rem}

We are now ready to establish the nonlocal consistency of the penalty scheme.

\begin{proof}[Nonlocal consistency proof]
  Let $\varphi \in C^{1,2}([0, T] \times \mathbb{R})$.
  For brevity, we define $\varphi_j^n \coloneqq \varphi(n \Delta t, j
  \Delta x)$ and $\varphi^n = (\varphi_0^n, \ldots, \varphi_M^n)^\intercal$.
  Let $(w^{\rho})_{\rho > 0}$ be a family of uniformly bounded
  real-valued maps from $[0, T] \times \mathbb{R}$ with half-relaxed
  limits $\overline{w}$ and $\underline{w}$.
  Let $(\rho_m, s_m, y_m, \xi_m)_m$ be an arbitrary sequence satisfying
  \begin{equation}
    \rho_m \downarrow 0
    \text{, }
    (s_m, y_m) \rightarrow (t, x)
    \text{, and }
    \xi_m \rightarrow 0
    \mathas m \rightarrow \infty.
    \label{eq:penalty_consistency_0}
  \end{equation}
  By the convention for off-grid evaluation of $S$ and
  $\mathcal{N}^\rho$, the off-grid case is handled by the same
  argument after translating the stencils.
  To avoid extra notation, we write the calculation for grid-point
  sequences satisfying $s_m = n_m \Delta t$ and $y_m = j_m \Delta x$
  (we have omitted the dependence of $\Delta t$ and $\Delta x$ on
  $\rho_m$ in the notation).
  For brevity, we define $\alpha \coloneqq (t, x)$ and $\alpha_m
  \coloneqq (s_m, y_m)$.

  By \cref{eq:penalty_scheme_1} and \cref{eq:penalty_scheme_2},
  \begin{equation}
    S(
      \rho_m,
      \alpha_m,
      \varphi(\alpha_m) + \xi_m,
      [\varphi + \xi_m]_{\alpha_m},
      \mathcal{N}^{\rho_m} w^{\rho_m}(\alpha_m)
    )
    =
    \begin{cases}
      \min ( S_m^{(1)}, S_m^{(2)} {+} \epsilon S_m^{(1)} )
      & \hspace{-0.65em} \text{if } n_m < N \\
      S_m^{(3)}
      & \hspace{-0.65em} \text{if } n_m = N
    \end{cases}
    \label{eq:penalty_consistency_1}
  \end{equation}
  where
  \begin{align}
    S_m^{(1)}
    & \coloneqq - \max_{b \in B^{\rho_m}} \left \{
      \frac{ \varphi_{j_m}^{n_m+1} - \varphi_{j_m}^{n_m} }{ \Delta t }
      + (L_b^{\rho_m} \varphi^{n_m})_{j_m}
      + f(\alpha_m, b)
    \right \}
    \nonumber \\
    S_m^{(2)}
    & \coloneqq \varphi_{j_m}^{n_m}
    + \xi_m - (\mathcal{M}^{\rho_m} w^{n_m, \rho_m})_{j_m}
    \nonumber \\
    S_m^{(3)}
    & \coloneqq \varphi_{j_m}^{n_m} + \xi_m - g(y_m).
    \label{eq:penalty_consistency_2}
  \end{align}
  In the above, we have used the notation $w^{n, \rho}$ of
  \cref{lem:intervention_limits}.
  Now, by \cref{lem:cluster_point},
  \begin{multline}
    \lim_{m \rightarrow \infty} S_m^{(1)}
    = - \lim_{m \rightarrow \infty} \max_{b \in B^{\rho_m}} \left \{
      \varphi_t(\alpha_m)
      + L_b \varphi(\alpha_m)
      + f(\alpha_m, b)
      + o(1)
    \right \}
    \\
    = - \max_{b \in B} \left \{
      \varphi_t(\alpha)
      + L_b \varphi(\alpha)
      + f(\alpha, b)
    \right \}.
    \label{eq:penalty_consistency_3}
  \end{multline}
  Moreover, by \cref{lem:intervention_limits},
  \begin{equation}
    \varphi(\alpha) - \mathcal{M} \overline{w}(\alpha)
    \leq \liminf_{m \rightarrow \infty} S_m^{(2)}
    \leq \limsup_{m \rightarrow \infty} S_m^{(2)}
    \leq \varphi(\alpha) - \mathcal{M} \underline{w}(\alpha)
    \label{eq:penalty_consistency_4}
  \end{equation}

  Suppose now that $t < T$.
  Since $s_m \rightarrow t$, we may assume $s_m < T$ (or,
  equivalently, $n_m < N$) for each $m$.
  In this case, taking limit inferiors of both sides of
  \cref{eq:penalty_consistency_1} and applying \cref{lem:commute} yields
  \begin{multline*}
    \liminf_{m \rightarrow \infty} S(
      \rho_m,
      \alpha_m,
      \varphi(\alpha_m) + \xi_m,
      [\varphi + \xi_m]_{\alpha_m},
      \mathcal{N}^{\rho_m} w^{\rho_m}(\alpha_m)
    )
    \\
    = \liminf_{m \rightarrow \infty} \min (
      S_m^{(1)}, S_m^{(2)} + \epsilon S_m^{(1)}
    )
    \geq \min \left (
      \lim_{m \rightarrow \infty} S_m^{(1)},
      \liminf_{m \rightarrow \infty} S_m^{(2)}
    \right ).
  \end{multline*}
  Applying \cref{eq:penalty_consistency_3} and
  \cref{eq:penalty_consistency_4} to the above,
  \begin{multline}
    \liminf_{m \rightarrow \infty} S(
      \rho_m,
      \alpha_m,
      \varphi(\alpha_m) + \xi_m,
      [\varphi + \xi_m]_{\alpha_m},
      \mathcal{N}^{\rho_m} w^{\rho_m}(\alpha_m)
    ) \\
    \geq \min \left (
      - \max_{b \in B} \left \{
        \varphi_t(\alpha)
        + L_b \varphi(\alpha)
        + f(\alpha, b)
      \right \},
      \varphi(\alpha) - \mathcal{M} \overline{w}(\alpha)
    \right ) \\
    = F_*(
      \alpha,
      \varphi(\alpha),
      D \varphi(\alpha),
      D^2 \varphi(\alpha),
      \mathcal{M} \overline{w}(\alpha)
    )
    \label{eq:penalty_consistency_5}
  \end{multline}
  where $D^2 \varphi \equiv \varphi_{xx}$, $D \varphi \equiv
  (\varphi_t, \varphi_x)$, and $F$ is given by \cref{eq:hjbqvi_operator}.
  In establishing the last equality in the above, we have used the
  fact that $F$ is continuous away from $t = T$ and hence $F = F_*$ there.
  Now, since $(\rho_m, s_m, y_m, \xi_m)_m$ is an arbitrary sequence
  satisfying \cref{eq:penalty_consistency_0},
  \cref{eq:penalty_consistency_5} implies
  \[
    \liminf_{\substack{
        \rho \downarrow 0 \\
        \beta \rightarrow \alpha \\
        \xi \rightarrow 0
    }} S(
      \rho,
      \beta,
      \varphi(\beta) + \xi,
      [\varphi + \xi]_{\beta},
      \mathcal{N}^\rho w^\rho(\beta)
    )
    \geq F_*(
      \alpha,
      \varphi(\alpha),
      D \varphi(\alpha),
      D^2 \varphi(\alpha),
      \mathcal{M} \overline{w}(\alpha)
    ),
  \]
  which is exactly the nonlocal consistency inequality
  \cref{eq:subconsistency} in the time-dependent case ($\alpha = (t, x)$).
  Symmetrically, we can establish the inequality
  \[
    \limsup_{\substack{
        \rho \downarrow 0 \\
        \beta \rightarrow \alpha \\
        \xi \rightarrow 0
    }} S(
      \rho,
      \beta,
      \varphi(\beta) + \xi,
      [\varphi + \xi]_{\beta},
      \mathcal{N}^\rho w^\rho(\beta)
    )
    \leq F^*(
      \alpha,
      \varphi(\alpha),
      D \varphi(\alpha),
      D^2 \varphi(\alpha),
      \mathcal{M} \underline{w}(\alpha)
    ),
  \]
  which corresponds to the nonlocal consistency inequality
  \cref{eq:superconsistency}.

  Suppose now that $t = T$.
  Since $s_m \rightarrow t$, it is possible that $s_m = T$ (or,
  equivalently, $n_m = N$) for one or more indices $m$ in the sequence.
  Let
  \[
    R_m^\pm \coloneqq
    \begin{cases}
      \min \left \{
        S_m^{(1)},
        S_m^{(2)} + \epsilon S_m^{(1)}
      \right \},
      & \text{if } n_m < N, \\
      \pm \infty,
      & \text{if } n_m = N.
    \end{cases}
  \]
  Therefore, by \cref{eq:penalty_consistency_1}
  \begin{equation}
    S(
      \rho_m,
      \alpha_m,
      \varphi(\alpha_m) + \xi_m,
      [\varphi + \xi_m]_{\alpha_m},
      \mathcal{N}^{\rho_m} w^{\rho_m}(\alpha_m)
    )
    \geq \min (
      R_m^+,
      S_m^{(3)}
    ).
    \label{eq:penalty_consistency_8}
  \end{equation}
  An immediate consequence of the definition of $S_m^{(3)}$ in
  \cref{eq:penalty_consistency_2} is that
  \begin{equation}
    \lim_{m \rightarrow \infty} S_m^{(3)}
    = \varphi(\alpha) - g(x)
    \geq \min \left(
      \varphi(\alpha) - g(x),
      \varphi(\alpha) - \mathcal{M} \overline{w}(\alpha)
    \right).
    \label{eq:penalty_consistency_9}
  \end{equation}
  Taking limit inferiors in \cref{eq:penalty_consistency_8} and
  applying the definition of $R_m^+$, \cref{lem:commute},
  \cref{eq:penalty_consistency_3}, \cref{eq:penalty_consistency_4},
  and \cref{eq:penalty_consistency_9},
  \begin{multline*}
    \liminf_{m \rightarrow \infty} S(
      \rho_m,
      \alpha_m,
      \varphi(\alpha_m) + \xi_m,
      [\varphi + \xi_m]_{\alpha_m},
      \mathcal{N}^{\rho_m} w^{\rho_m}(\alpha_m)
    ) \\
    \geq \min \Bigl (
      \min \left \{
        -\max_{b \in B} \left \{
          \varphi_t(\alpha)
          + L_b \varphi(\alpha)
          + f(\alpha, b)
        \right \},
        \varphi(\alpha) - \mathcal{M} \overline{w}(\alpha)
      \right \}, \\
      \min \left \{
        \varphi(\alpha) - g(x),
        \varphi(\alpha) - \mathcal{M} \overline{w}(\alpha)
      \right \}
    \Bigr )
    = F_*(
      \alpha,
      \varphi(\alpha),
      D \varphi(\alpha),
      D^2 \varphi(\alpha),
      \mathcal{M} \overline{w}(\alpha)
    ).
  \end{multline*}
  As in the previous paragraph, this implies the nonlocal consistency
  inequality \cref{eq:subconsistency}.

  It remains to establish \cref{eq:superconsistency} in the case of $t = T$.
  Since $\mathcal{M} g \leq g$ by \ref{enu:vanilla_start}, it follows
  that $g(y_m) = \max ( g(y_m), \mathcal{M} g(y_m) )$ for each $m$.
  By \cref{rem:wlog}, we can, without loss of generality, assume that
  $w^\rho$ is a solution of the scheme so that $w^\rho(N \Delta t, j
  \Delta x) = g(j \Delta x)$ for all $j$, corresponding to the
  terminal condition.
  Therefore,
  \begin{align*}
    S_m^{(3)}
    & = \varphi_{j_m}^{n_m} + \xi_m - \max \left( g(y_m), \mathcal{M}
    g(y_m) \right) \\
    & = \min \left(
      \varphi_{j_m}^{n_m} + \xi_m - g(y_m),
      \varphi_{j_m}^{n_m} + \xi_m - \mathcal{M} g(y_m)
    \right) \\
    & = \min \left(
      \varphi_{j_m}^{n_m} + \xi_m - g(y_m),
      \varphi_{j_m}^{n_m} + \xi_m - (\mathcal{M}^{\rho_m} w^{N,
      \rho_m})_{j_m} + r_m
    \right)
  \end{align*}
  where $r_m \rightarrow 0$ by the uniform continuity of $g$ on
  compact sets, the continuity of $\Gamma$ and $K$, and the local
  uniform Hausdorff convergence of $Z^\rho$ to $Z$.
  It follows that, letting
  \[
    S_m^{(4)} \coloneqq \min \left \{
      \varphi_{j_m}^{n_m} + \xi_m - g(y_m),
      \varphi_{j_m}^{n_m} + \xi_m - (\mathcal{M}^{\rho_m} w^{n_m,
      \rho_m})_{j_m} + r_m
    \right \},
  \]
  we have $S_m^{(3)} = S_m^{(4)}$ whenever $n_m = N$.
  Therefore, by \cref{eq:penalty_consistency_1},
  \begin{equation}
    S(
      \rho_m,
      \alpha_m,
      \varphi(\alpha_m) + \xi_m,
      [\varphi + \xi_m]_{\alpha_m},
      \mathcal{N}^{\rho_m} w^{\rho_m}(\alpha_m)
    ) {\leq}
    \max (
      R_m^-,
      S_m^{(4)}
    ).
    \label{eq:penalty_consistency_10}
  \end{equation}
  Moreover, by \cref{lem:commute} and \cref{lem:intervention_limits},
  \begin{equation}
    \limsup_{m \rightarrow \infty} S_m^{(4)}
    \leq \min \left \{
      \varphi(\alpha) - g(x),
      \varphi(\alpha) - \mathcal{M} \underline{w}(\alpha)
    \right \}.
    \label{eq:penalty_consistency_11}
  \end{equation}
  Taking limit superiors in \cref{eq:penalty_consistency_10} and
  applying the definition of $R_m^-$, \cref{lem:commute},
  \cref{eq:penalty_consistency_3}, \cref{eq:penalty_consistency_4},
  and \cref{eq:penalty_consistency_11},
  \begin{multline*}
    \limsup_{m \rightarrow \infty} S(
      \rho_m,
      \alpha_m,
      \varphi(\alpha_m) + \xi_m,
      [\varphi + \xi_m]_{\alpha_m},
      \mathcal{N}^{\rho_m} w^{\rho_m}(\alpha_m)
    ) \\
    \leq \max \Bigl (
      \min \left \{
        -\max_{b \in B} \left \{
          \varphi_t(\alpha)
          + L_b \varphi(\alpha)
          + f(\alpha, b)
        \right \},
        \varphi(\alpha) - \mathcal{M} \underline{w}(\alpha)
      \right \}, \\
      \min \left \{
        \varphi(\alpha) - g(x),
        \varphi(\alpha) - \mathcal{M} \underline{w}(\alpha)
      \right \}
    \Bigr )
    = F^*(
      \alpha,
      \varphi(\alpha),
      D \varphi(\alpha),
      D^2 \varphi(\alpha),
      \mathcal{M} \underline{w}(\alpha)
    ),
  \end{multline*}
  which establishes \cref{eq:superconsistency}, as desired.
\end{proof}

\subsubsection{Higher dimensions}
\label{subsec:higher_dimensions}

In higher dimensions, the operator $L_b$ may contain
cross-derivatives which, in the context of the optimal control
problem \cref{eq:control_problem}, corresponds to correlations
between the components of $X^{t, x, \gamma}$.
A na\"{i}ve discretization of these cross-derivatives (e.g., using a
standard five-point stencil) results in a nonmonotone scheme.
Unfortunately, nonmonotone schemes are well-known to be capable of
failing to converge to the viscosity solution \cite{MR2218974}.

It is possible to resolve these issues by discretizing $L_b$ using
wide-stencils similarly to the implicit scheme described in
\cite[Corollary 5.1]{MR3042570}.
This results in a scheme that is first order accurate so long as the
stencil is of length $O(\Delta x^{1/2})$.
Since wide-stencils are well-understood, we simply mention that it is
routine to extend the convergence analysis of this section to a
wide-stencil version of the penalty scheme.
In particular, it is trivial to extend \cref{lem:intervention_limits}
to higher dimensions by interpreting $j \equiv (j_1, \ldots, j_d)$ as
a multi-index and $\interp(\cdot)$ as a monotone multi-dimensional
linear interpolation (in this case, $x_j \equiv (x_{j_1}, \ldots,
  x_{j_d})$ would be a point on a rectilinear grid, e.g., $( \{ k
\Delta x \}_{j=-M}^M )^d$).

\subsubsection{Infinite-horizon (steady state) case}

If we assume the functions $f$, $Z$, $\Gamma$, and $K$ are
independent of time (i.e., $f(t,x,b)=f(x,b)$, etc.), the
infinite-horizon analogue of \cref{eq:hjbqvi} is given by
\begin{equation}
  \min \left (
    - \sup_{b \in B} \left \{
      L_b u - \beta u + f(\cdot, b)
    \right \},
    u - \mathcal{M} u
  \right ) = 0
  \text{ on } \mathbb{R}^d
  \label{eq:hjbqvi_elliptic}
\end{equation}
where $\beta>0$ is a positive discount factor.
Note that in the above, $u$, $L_b u$, and $\mathcal{M} u$ are no
longer functions of time and space but rather functions of space alone.
Specifically, we interpret $\mathcal{M}$ above as $\mathcal{M} u(x)
\coloneqq \sup_{z \in Z(x)} \{ u(x + \Gamma(x, z)) + K(x, z) \}$.

\ifx\myarticle\undefined
\else
As with \cref{eq:hjbqvi}, \cref{eq:hjbqvi_elliptic} can be
interpreted as the dynamic programming equation associated to the
optimal control problem whose value function is given by (compare
with \cref{eq:control_problem})
\[
  u(x)
  = \sup_\gamma \mathbb{E} \left[
    \int_0^\infty e^{-\beta t} f(X_t^{x, \gamma}, b_t) dt
    + \sum_{\tau_i} e^{-\beta \tau_i} K(X_{\tau_i-}^{x, \gamma}, z_i)
  \right]
\]
where
\[
  X_t^{x, \gamma}
  \coloneqq x
  + \int_0^t \mu(X_s^{x, \gamma}, b_s) ds
  + \int_0^t \sigma(X_s^{x, \gamma}, b_s) dW_s
  + \sum_{\tau_i \leq t} \Gamma(X_{\tau_i-}^{x, \gamma}, z_i).
\]
\fi

We can adapt the scheme \cref{eq:penalty} to this setting by
considering the discrete equations
\[
  - \max_{b \in B^\rho} \left \{
    (L_b^\rho \vec{u})_j
    - \beta u_j
    + f(x_j, b)
  \right \}
  - \left( \frac{ (\mathcal{M}^\rho \vec{u})_j - u_j }{\epsilon} \right)^+
  = 0
\]
for $-M \leq j \leq M$, where we have written $\vec{u} \coloneqq
(u_{-M}, \ldots, u_M)^\intercal$ to stress that the numerical
solution is a vector in $\mathbb{R}^{2M+1}$.
We interpret $\mathcal{M}^\rho$ in the above as $(\mathcal{M}^\rho
\vec{u})_j \coloneqq \sup_{z \in Z^\rho(x_j)} \{ \interp(\vec{u}, x_j
+ \Gamma(x_j, z)) + K(x_j, z) \}$, with $Z^\rho$ satisfying a
time-independent version of \ref{enu:comparison_end}.
\ifx\myarticle\undefined
The analysis of this scheme is nearly identical to that of \cref{eq:penalty}.
\else
The analysis of this scheme is identical to that of
\cref{eq:penalty}, save that we obtain the stability bound $\Vert u
\Vert_\infty \leq \beta^{-1} \Vert f \Vert_\infty$.
\fi

\subsection{Semi-Lagrangian scheme}

We now turn to a semi-Lagrangian scheme for the HJBQVI introduced in
\cite[Section 5.3]{MR3493959}, which can be used if the diffusion
coefficient $\sigma$ does not depend on the control (i.e., $\sigma(x,
b) = \sigma(x)$).
As we will review shortly, the advantage of this scheme is that it
only requires a single solution of a linear system per timestep
(compare this with the penalty scheme, which requires an expensive
iterative method; see \cref{rem:iterative_method}).
As usual, we focus only on the one-dimensional case ($d=1$) since an
extension to higher dimensions can be handled by wide-stencils (see
\cref{subsec:higher_dimensions}).
Roughly speaking, the semi-Lagrangian scheme
\begin{enumerate}
  \item
    discretizes first order terms by tracing the path of a
    deterministic particle as determined by the \emph{Lagrangian
    derivative} $\mathbb{D}$ (defined below) and
  \item
    discretizes the intervention operator using information from the
    $(n + 1)$-th timestep, adding a $O(\Delta t)$ term in order to
    ``linearize'' the scheme.
\end{enumerate}

We start by discussing the first point (from which the scheme's name
is derived).
Let $X^{x, b}$ denote a particle whose position at time $s$ is given by
\[
  X_s^{x, b} \coloneqq x + \int_0^s \mu(X_v^{x, b}, b) dv
\]
(in the above, $b$ is a fixed element of $B$; not a process taking
values in $B$).
For a ``smooth enough'' function $u$ of time and space, we consider
the Lagrangian derivative
\begin{multline}
  \mathbb{D}_b u(t, x)
  \coloneqq
  \frac{\partial u}{\partial t}(t, x)
  + \left \langle \mu(x, b), D_x u(t, x) \right \rangle
  = \lim_{s \downarrow 0} \frac{
    u(t + s, X_s^{x, b}) - u(t, x)
  }{s}
  \\
  \approx \frac{
    u(t + \Delta t, x + \mu(x, b) \Delta t) - u(t, x)
  }{\Delta t}
  \label{eq:lagrangian_derivative}
\end{multline}
where the approximation above is justified for $\Delta t > 0$ small
enough since in this case, we expect $X_{\Delta t}^{x, b} \approx x +
\mu(x, b) \Delta t$.
Substituting $t = n \Delta t$ and $x = x_j$ into
\cref{eq:lagrangian_derivative} and using interpolation to
approximate the term $u(t + \Delta t, x + \mu(x, b) \Delta t)$ yields
\begin{equation}
  \frac{\partial u}{\partial t}(n \Delta t, x_j)
  + \mu_j(b) u_x(n \Delta t, x_j)
  \approx \frac{
    \interp(u^{n+1}, x_j + \mu_j(b) \Delta t) - u_j^n
  }{\Delta t}.
  \label{eq:sls_motivation_1}
\end{equation}

As for the second point, we take
\begin{equation}
  \mathcal{M} u(n \Delta t, x_j)
  \approx (\mathcal{M}^\rho u^{n + 1})_j
  + \frac{1}{2} \sigma_j^2 (\mathcal{D}^2 u^n)_j \Delta t
  \label{eq:sls_motivation_2}
\end{equation}
as a discretization of the intervention operator, where we have used
the shorthand $\sigma_j \coloneqq \sigma(x_j)$.
As we will see shortly, the $O(\Delta t)$ term in
\cref{eq:sls_motivation_2} allows us to express the scheme as a
linear system at each timestep (at the cost of $O(\Delta t)$
discretization error).

Finally, we substitute \cref{eq:sls_motivation_1} and
\cref{eq:sls_motivation_2} into the HJBQVI \cref{eq:hjbqvi} to arrive
at the semi-Lagrangian scheme:
\[
  \begin{cases}
    \min \left (
      \begin{gathered}
        - \displaystyle{ \max_{b \in B^\rho} } \biggl \{
          \displaystyle{
            \frac{
              \interp(u^{n+1}, x_j + \mu_j(b) \Delta t)
              - u_j^n
            }{\Delta t}
          }
          + \frac{1}{2} \sigma_j^2 (\mathcal{D}^2 u^n)_j
          + f_j^n(b)
        \biggr \},
        \\
        u_j^n - (\mathcal{M}^\rho u^{n + 1})_j
        - \displaystyle{ \frac{1}{2} } \sigma_j^2 (\mathcal{D}^2 u^n)_j \Delta t
      \end{gathered}
    \right ) = 0,
    \\
    \hspace{30.5em} \text{for } n < N
    \\
    u_j^N = g(x_j)
  \end{cases}
\]
By some simple algebra, the discrete equations above are equivalent to
\[
  \begin{cases}
    (A u^n)_j
    = \max \left (
      \displaystyle{ \max_{b \in B^\rho} } \left \{
        \interp(u^{n+1}, x_j + \mu_j(b) \Delta t)
        + f_j^n(b) \Delta t
      \right \},
      (\mathcal{M}^\rho u^{n+1})_j
    \right ),
    \\
    \hspace{29.5em} \text{for } n < N
    \\
    u_j^N = g(x_j)
  \end{cases}
\]
where $(A u^n)_j \coloneqq u_j^n - \frac{1}{2} \sigma_j^2
(\mathcal{D}^2 u^n)_j \Delta t$.
Note that we can also represent $A$ as the matrix
\[
  A =
  I+\frac{\Delta t}{2(\Delta x)^2}
  \begin{pmatrix}
    0\\
    -(\sigma_{1})^{2} & 2(\sigma_{1})^{2} & -(\sigma_{1})^{2}\\
    & -(\sigma_{2})^{2} & 2(\sigma_{2})^{2} & -(\sigma_{2})^{2}\\
    &  & \ddots & \ddots & \ddots\\
    &  &  & -(\sigma_{M-1})^{2} & 2(\sigma_{M-1})^{2} & -(\sigma_{M-1})^{2}\\
    &  &  &  &  & 0
  \end{pmatrix},
\]
interpreting $(A u^n)_j$ as the $j$-th entry of the matrix-vector
product $A u^n$.
In particular, the matrix $A$ is strictly diagonally dominant and
hence nonsingular.

\begin{theorem}
  \label{thm:sls_convergence}
  If the diffusion coefficient $\sigma$ does not depend on the
  control (i.e., $\sigma(x, b) = \sigma(x)$) and the assumptions of
  \cref{thm:penalty_convergence} are satisfied, the solution of the
  semi-Lagrangian scheme converges locally uniformly to the unique
  bounded solution of the HJBQVI.
\end{theorem}

As with the penalty scheme, we write the semi-Lagrangian scheme in
the form \cref{eq:scheme} by taking
\begin{multline*}
  S(\rho, (n \Delta t, j \Delta x), u_j^n, [u]_j^n, \ell) \coloneqq
  \\
  \begin{cases}
    \min \left (
      \begin{gathered}
        - \displaystyle{ \max_{b \in B^\rho} } \biggl \{
          \displaystyle{
            \frac{
              \interp(u^{n+1}, x_j + \mu_j(b) \Delta t)
              - u_j^n
            }{\Delta t}
          }
          + \frac{1}{2} \sigma_j^2 (\mathcal{D}^2 u^n)_j
          + f_j^n(b)
        \biggr \},
        \\
        u_j^n - \ell
        - \displaystyle{ \frac{1}{2} } \sigma_j^2 (\mathcal{D}^2 u^n)_j \Delta t
      \end{gathered}
    \right ),
    \\
    \hspace{29.5em} \text{if } n < N
    \\
    u_j^n - g(x_j),
    \hspace{24.4em} \text{if } n = N
  \end{cases}
\end{multline*}
and
\[
  \mathcal{N}^\rho u(n \Delta t, j \Delta x)
  \coloneqq (\mathcal{M}^\rho u^{n + 1})_j.
\]
\ifx\myarticle\undefined
As usual, the stability and monotonicity arguments, which appear in
the preprint version
\href{https://arxiv.org/abs/1705.02922}{arXiv:1705.02922}, are
standard and hence omitted.
Below, we give only a proof of nonlocal consistency.
\else
Below, we handle the stability, monotonicity, and nonlocal
consistency of the semi-Lagrangian scheme to establish
\cref{thm:sls_convergence}.

\subsubsection{Monotonicity}

\begin{proof}[Monotonicity proof]
  With $v$ defined as in \cref{subsubsec:penalty_monotonicity}, the
  analogue of \cref{eq:penalty_monotonicity} is
  \begin{multline*}
    S(\rho, (n \Delta t, x_j), u_j^n, [u]_j^n, \ell)
    - S(\rho, (n \Delta t, x_j), w_j^n, [w]_j^n, \ell)
    \\
    \leq \max \left \{
      \max_{b \in B^\rho} \left \{
        \frac{\interp(v^{n + 1}, x_j + \mu_j(b) \Delta t)}{\Delta t}
        + \frac{1}{2} \sigma_j^2 (\mathcal{D}^2 v^n)_j
      \right \},
      \frac{1}{2} \sigma_j^2 (\mathcal{D}^2 v^n)_j \Delta t
    \right \}
    \leq 0.
  \end{multline*}
  In the above, we have used the fact that the terms
  \[
    \interp(v^{n + 1}, x_j + \mu_j(b) \Delta t)
    = \alpha v_{k+1}^{n + 1} + \left(1 - \alpha\right) v_k^{n + 1}
  \]
  and $(\mathcal{D}^2 v^n)_j$ are nonpositive.
\end{proof}

\subsubsection{Stability}

Similar to the penalty scheme, we obtain a bound of $\Vert g
\Vert_\infty + \Vert f \Vert_\infty T$ on the solution.

\begin{proof}[Stability proof]
  Let $u$ be the solution of the semi-Lagrangian scheme.
  Fix a particular timestep $n < N$.
  Note that for any $j$, $\min_j u_j^{n + 1} \leq \interp(u^{n+1},
  \cdot) \leq \max_j u_j^{n + 1}$ and $(\mathcal{M}^\rho u^{n+1})_j
  \leq \max_j u_j^{n + 1}$ (since $K \leq 0$ by \ref{enu:vanilla_end}).
  Now, let $j-$ and $j+$ be defined as in \cref{subsubsec:penalty_stability}.
  It is readily verified that $(\mathcal{D}^2 u^n)_{j-} \geq 0$ and
  $(\mathcal{D}^2 u^n)_{j+} \leq 0$.
  Therefore,
  \begin{multline*}
    u_{j-}^n
    \geq u_{j-}^n - \frac{1}{2} \sigma_{j-}^2 (\mathcal{D}^2 u^n)_{j-} \Delta t
    = (A u^n)_{j-}
    \\
    = \max \left (
      \max_{b \in B^\rho} \left \{
        \interp(u^{n+1}, x_{j-} + \mu_{j-}(b) \Delta t)
        + f_{j-}^n(b) \Delta t
      \right \},
      (\mathcal{M}^\rho u^{n+1})_{j-}
    \right )
    \\
    \geq \max_{b \in B^\rho} \left \{
      \interp(u^{n+1}, x_{j-} + \mu_{j-}(b) \Delta t)
      + f_{j-}^n(b) \Delta t
    \right \}
    \geq \min_j u_j^{n + 1} - \Vert f \Vert_\infty \Delta t
  \end{multline*}
  and
  \begin{multline*}
    u_{j+}^n
    \leq u_{j+}^n - \frac{1}{2} \sigma_{j+}^2 (\mathcal{D}^2 u^n)_{j+} \Delta t
    = (A u^n)_{j+}
    \\
    = \max \left (
      \max_{b \in B^\rho} \left \{
        \interp(u^{n+1}, x_{j+} + \mu_{j+}(b) \Delta t)
        + f_{j+}^n(b) \Delta t
      \right \},
      (\mathcal{M}^\rho u^{n+1})_{j+}
    \right )
    \\
    \leq \max \left(
      \max_j u_j^{n + 1}
      + \Vert f \Vert_\infty \Delta t
      ,
      \max_j u_j^{n + 1}
    \right)
    \leq \max_j u_j^{n + 1} + \Vert f \Vert_\infty \Delta t.
  \end{multline*}
  The remainder of the proof proceeds as in \cref{subsubsec:penalty_stability}.
\end{proof}

\subsubsection{Nonlocal consistency}

\fi

\begin{proof}[Nonlocal consistency proof]
  Let $\varphi \in C^{1,2}([0, T] \times \mathbb{R})$.
  For brevity, we define $\varphi_j^n \coloneqq \varphi(n \Delta t, j
  \Delta x)$ and $\varphi^n = (\varphi_0^n, \ldots, \varphi_M^n)^\intercal$.
  Let $(\rho_m, s_m, y_m, \xi_m)_m$ be an arbitrary sequence
  satisfying \cref{eq:penalty_consistency_0}; by the convention for
  off-grid evaluation of $S$ and $\mathcal{N}^\rho$, we write the
  calculation for the case $s_m = n_m \Delta t$ and $y_m = j_m \Delta x$.
  For indices with $n_m < N$, we have
  \begin{multline}
    \frac{
      \interp(
        \varphi^{n_m + 1} + \xi_m,
        y_m + \mu_{j_m}(b) \Delta t
      ) - \varphi(s_m, y_m) - \xi_m
    }{\Delta t}
    \\
    = \frac{
      \varphi(s_m + \Delta t, y_m + \mu_{j_m}(b) \Delta t)
      - \varphi(s_m, y_m) + \xi_m - \xi_m
    }{\Delta t}
    + O \left( \frac{(\Delta x)^2}{\Delta t} \right)
    \\
    = \varphi_t(t, x) + \mu(x, b) \varphi_x(t, x)
    + O \left( \frac{(\Delta x)^2}{\Delta t} \right) + o(1)
    = \varphi_t(t, x) + \mu(x, b) \varphi_x(t, x)
    + o(1)
    \label{eq:sls_consistency_1}
  \end{multline}
  by a Taylor expansion.
  In the above, we used the fact that since $Q \rightarrow \infty$ as
  $\rho \downarrow 0$, we can always find $m$ large enough such that
  $y_m + \mu_{j_m}(b) \Delta t$ is a point inside $[-Q, Q]$.

  The remaining arguments are a slight modification of those for the
  penalty scheme.
  In particular, we need only modify the proof by setting $\epsilon =
  0$ in \cref{eq:penalty_consistency_1} and replacing the definitions
  of $S_m^{(1)}$ and $S_m^{(2)}$ in \cref{eq:penalty_consistency_2} by
  \begin{multline*}
    \hat{S}_m^{(1)}
    = -\max_{b \in B^{\rho_m}} \biggl \{
      \frac{
        \interp(
          \varphi^{n_m + 1} + \xi_m,
          y_m + \mu_{j_m}(b) \Delta t
        ) - \varphi(s_m, y_m) - \xi_m
      }{\Delta t}
      \\
      + \frac{1}{2} \sigma_{j_m}^2 (\mathcal{D}^2 \varphi^{n_m})_{j_m}
      + f_{j_m}^{n_m}(b)
    \biggr \}
  \end{multline*}
  and
  \[
    \hat{S}_m^{(2)}
    \coloneqq \varphi_{j_m}^{n_m} + \xi_m - (\mathcal{M}^{\rho_m}
    w^{n_m + 1, \rho_m})_{j_m}
    - \frac{1}{2} \sigma_{j_m}^2 (\mathcal{D}^2 \varphi^{n_m})_{j_m} \Delta t
  \]
  where we have used the symbol $\hat{\cdot}$ to distinguish new
  definitions from the old.
  By \cref{eq:sls_consistency_1},
  \[
    \lim_{m \rightarrow \infty} \hat{S}_m^{(1)} = \lim_{m \rightarrow
    \infty} S_m^{(1)}.
  \]
  Moreover, by \cref{lem:intervention_limits},
  \[
    \varphi(t, x) - \mathcal{M} \overline{w}(t, x)
    \leq \liminf_{m \rightarrow \infty} \hat{S}_m^{(2)}
    \leq \limsup_{m \rightarrow \infty} \hat{S}_m^{(2)}
    \leq \varphi(t, x) - \mathcal{M} \underline{w}(t, x)
  \]
  (compare with \eqref{eq:penalty_consistency_4}).
\end{proof}

\bibliography{main}

\appendix

\ifx\myarticle\undefined
\clearpage
\else
\fi

\section{Comparison principle for the HJBQVI}
\label{app:comparison_principle}

We now give a comparison principle for the HJBQVI.
Recall that the HJBQVI is given by \cref{eq:pde} when
$\overline{\Omega} = [0, T] \times \mathbb{R}^d$ and $\mathcal{M}$
and $F$ are defined in \cref{eq:intervention} and \cref{eq:hjbqvi_operator}.
Throughout this appendix, we assume \ref{enu:vanilla_start}\textendash
\ref{enu:comparison_end}.
\ref{enu:comparison_end} is only needed to obtain the
Hausdorff-continuity of $Z$; the discrete control sets are immaterial to the
HJBQVI.

\begin{proposition}
  Suppose $\mu$ and $\sigma$ are Lipschitz in $x$ uniformly in $b$ (i.e.,
    $
    \left| \mu(x,b) - \mu(y,b) \right|
    + \left| \sigma(x,b) - \sigma(y,b) \right|
    \leq \const \left| x - y \right|
  $).
  Then, the HJBQVI satisfies a
  \cref{eq:nonlocal_definition}-comparison principle in $B([0, T]
  \times \mathbb{R}^d)$ (for any $d \geq 1$).
\end{proposition}

A proof of the above is found in \cite[Theorem
2.13]{azimzadeh2017zero}; see also \cite[Theorem 5.11]{MR2568293} for
a related comparison principle for solutions of polynomial growth.
However, in both these works, the notion of viscosity solution used
is the following:

\begin{definition}
  \label{def:viscosity_solution_3}
  $u \in B_{\loc}([0, T] \times \mathbb{R}^d)$ is a subsolution of
  \cref{eq:pde} if for all $\varphi \in C^{1,2}([0, T) \times \mathbb{R}^d)$ and
  $(t, x) \in [0, T) \times \mathbb{R}^d$ such that $u^* - \varphi$ has a local
  maximum at $(t, x)$, we have
  \[
    F((t, x), u^*(t, x), D \varphi(t, x), D_x^2 \varphi(t, x),
    \mathcal{M} u^*(t, x)) \leq 0
  \]
  and, for all $x \in \mathbb{R}^d$,
  \[
    \min \left\{ u^*(T, x) - g(x), u^*(T, x) - \mathcal{M}u^*(T, x)
    \right\} \leq 0.
  \]
  Supersolutions are defined symmetrically.
\end{definition}

Unlike \cref{def:viscosity_solution}, the above does not use the
envelopes $F_*$ and $F^*$.
Fortunately, for the HJBQVI, these solution concepts are equivalent:

\begin{proposition}
  For the HJBQVI, Definitions \ref{def:viscosity_solution} and
  \ref{def:viscosity_solution_3} are equivalent.
\end{proposition}

We prove only the nontrivial direction.
The proof is similar to \cite[Remark 3.2]{MR2857450}.

\begin{proof}
  Let $u$ be a subsolution in the sense of \cref{def:viscosity_solution}.
  Without loss of generality, assume $u$ is USC.
  Since $F = F_*$ away from the terminal time, we need only establish
  the terminal inequality.
  To this end, fix $x_0\in\mathbb{R}^{d}$ and suppose, to arrive at
  a contradiction,
  \[
    \min\left\{
      u(T, x_0) - g(x_0),
      (u - \mathcal{M}u)(T,x_0)
    \right\} > 0.
  \]
  Consider the cylinder $C \coloneqq [0,T] \times \overline{B(x_0, 1)}$.
  For $\alpha > 0$, let
  \[
    \psi_\alpha(t, x)
    \coloneqq \alpha \left| x - x_0 \right|^2 + c_\alpha \left( T - t \right)
  \]
  where $c_\alpha > 0$ is chosen such that
  $c_\alpha \rightarrow \infty$ as $\alpha \rightarrow \infty$ and
  \[
    c_\alpha
    \geq \alpha + \sup_{(t,x,b) \in C \times B} \left\{
      \alpha L_b \left| x - x_0 \right|^2 + \left| f(t,x,b) \right|
    \right\}.
  \]
  Since $u$ is USC, we can find a maximizer $(t_\alpha,x_\alpha)$
  of $u - \psi_\alpha$ on $C$. Since $(T, x_0)$ is in $C$,
  \[
    (u - \psi_\alpha)(t_\alpha,x_\alpha)
    \geq (u - \psi_\alpha)(T, x_0)
    = u(T, x_0).
  \]
  It follows that $\psi_\alpha(t_\alpha, x_\alpha)$ is bounded
  by a constant independent of $\alpha$ because
  \[
    0
    \leq \psi_\alpha(t_\alpha, x_\alpha)
    \leq u(t_\alpha, x_\alpha) - u(T, x_0)
    \leq \sup_C u - u(T, x_0).
  \]
  Therefore, $(t_\alpha, x_\alpha) \rightarrow (T, x_0)$.
  Let
  \[
    \varphi_\alpha
    \coloneqq \psi_\alpha + (u - \psi_\alpha)(t_\alpha, x_\alpha).
  \]
  For $\alpha$ sufficiently large, $u - \varphi_\alpha$ has a local
  maximum at $(t_\alpha, x_\alpha)$ on $[0,T] \times \mathbb{R}^d$
  and hence we obtain the subsolution inequality
  \begin{multline*}
    0\geq
    F_*(
      (t_\alpha, x_\alpha),
      u(t_\alpha, x_\alpha),
      D\varphi_\alpha(t_\alpha, x_\alpha),
      D^2_x\varphi_\alpha(t_\alpha, x_\alpha),
      \mathcal{M}u(t_\alpha, x_\alpha)
    ) \\
    \geq \min\left\{
      \alpha,
      (u - \mathcal{M}u)(t_\alpha, x_\alpha),
      u(t_\alpha, x_\alpha) - g(x_\alpha)
    \right\}.
  \end{multline*}
  Since
  \[
    u(t_\alpha, x_\alpha)\geq
    (u - \psi_\alpha)(t_\alpha, x_\alpha)
    \geq u(T,x_0),
  \]
  taking limit inferiors of both sides yields
  $\liminf_{\alpha \rightarrow \infty} u(t_\alpha, x_\alpha) \geq u(T, x_0)$.
  Since $u$ is USC,
  $\limsup_{\alpha \rightarrow \infty} u(t_\alpha, x_\alpha) \leq u(T,x_0)$.
  Therefore, $u(t_\alpha, x_\alpha)\rightarrow u(T, x_0)$ and hence
  \[
    u(t_\alpha, x_\alpha) - g(x_\alpha) \rightarrow u(T, x_0) - g(x_0) > 0.
  \]
  Since $\mathcal{M}u$ is USC by \cite[Lemma 4.3]{MR2568293},
  \[
    \liminf_{\alpha\rightarrow\infty} (u - \mathcal{M}u)(t_\alpha, x_\alpha)
    \geq (u - \mathcal{M}u)(T, x_0)
    > 0.
  \]
  It follows that, for $\alpha$ sufficiently large,
  \[
    \left(u - \mathcal{M}u\right)(t_\alpha,x_\alpha) > 0
    \text{ and }
    u(t_\alpha,x_\alpha) - g(x_\alpha) > 0,
  \]
  contradicting the subsolution inequality.

  The case of supersolutions is handled similarly.
  Let $u$ be an LSC supersolution in the sense of
  \cref{def:viscosity_solution} and assume
  \[
    \min\left\{
      u(T, x_0) - g(x_0),
      (u - \mathcal{M}u)(T,x_0)
    \right\} < 0.
  \]
  Define $C$, $\psi_\alpha$, and $c_\alpha$ as in the previous proof.
  Let $(t_\alpha, x_\alpha)$ be minimizer of $u + \psi_\alpha$ on $C$
  and define the test function
  \[
    \varphi_\alpha
    \coloneqq - \psi_\alpha + (u + \psi_\alpha)(t_\alpha, x_\alpha).
  \]
  As in the previous proof, we obtain $(t_\alpha, x_\alpha)
  \rightarrow (T, x_0)$ and $u(t_\alpha, x_\alpha) \rightarrow u(T,
  x_0)$ and, for $\alpha$ sufficiently large, the supersolution inequality
  \begin{multline*}
    0\leq F^*(
      (t_\alpha,x_\alpha),
      u(t_\alpha,x_\alpha),
      D\varphi_\alpha(t_\alpha,x_\alpha),
      D^2_x\varphi_\alpha(t_\alpha, x_\alpha),
      \mathcal{M}u(t_\alpha,x_\alpha)
    ) \\
    \leq \max\left\{
      -\alpha,
      \min\left\{
        (u-\mathcal{M}u)(t_\alpha, x_\alpha),
        u(t_\alpha, x_\alpha) - g(x_\alpha)
      \right\}
    \right\}.
  \end{multline*}
  Since $\mathcal{M}u$ is LSC by \cite[Lemma 4.3]{MR2568293} and $g$
  is continuous, \cref{lem:commute} yields
  \begin{multline*}
    \limsup_{\alpha \rightarrow \infty} \min \left\{
      u(t_\alpha,x_\alpha)-g(x_\alpha),
      (u - \mathcal{M}u)(t_\alpha,x_\alpha)
    \right\} \\
    \leq \min \left\{
      u(T,x_0) - g(x_0),
      (u - \mathcal{M}u)(T,x_0)
    \right\}.
  \end{multline*}
  It follows that, for $\alpha$ sufficiently large,
  \[
    \min \left\{
      u(t_\alpha,x_\alpha)-g(x_\alpha),
      (u - \mathcal{M}u)(t_\alpha,x_\alpha)
    \right\} < 0,
  \]
  contradicting the supersolution inequality.
\end{proof}

\ifx\myarticle\undefined
\else
\section{Overstepping error}

We close with a discussion which, while not related to the
theoretical convergence of the semi-Lagrangian scheme
(\cref{thm:sls_convergence}), is of great importance for a practical
implementation.

As discussed in \cref{rem:truncated_domain}, a practical
implementation should take $Q = \const$
In this case, the point $x_j + \mu_j(b) \Delta t$ may not lie in the
numerical domain $[-Q, Q]$ regardless of how small we take $\rho$.
Therefore, if $Q = \const$, the equality \cref{eq:sls_consistency_1}
fails to hold, and the semi-Lagrangian approximation introduces
$O(1)$ ``overstepping error''.
Semi-Lagrangian schemes are well-known to exhibit this phenomenon
(see, e.g., \cite{MR3661103,picarelli2017boundary}).

One way to avoid overstepping error is to modify the spatial grid so
that the distance between boundary nodes approaches zero sublinearly
as $\rho \downarrow 0$.
For example, we may use a spatial grid satisfying (using
Bachmann-Landau notation)
\begin{align}
  x_{j + 1} - x_j
  & = \Theta( \rho )
  & \text{for } -M < j < M-1
  \nonumber \\
  x_{j + 1} - x_j
  & = o( \sqrt{ \rho } ) \cap \omega( \rho )
  & \text{for } j = -M, M-1.
  \label{eq:altered_grid}
\end{align}
We sketch the idea below.

Let
\[
  \Vert \mu \Vert_{[-Q, Q]}
  \coloneqq \max_{(x, b) \in [-Q, Q] \times B} \left| \mu(x, b) \right|
  < \infty.
\]
If $x_j$ is a grid point satisfying $x_j < Q$, then
\begin{align*}
  x_j + \mu_j(b) \Delta t
  & \leq x_{M-1} + \Vert \mu \Vert_{[-Q, Q]} \Delta t \\
  & = x_M - \left( x_M - x_{M-1} \right) + \Vert \mu \Vert_{[-Q, Q]} \Delta t \\
  & = Q - \left( x_M - x_{M-1} \right) + \const \rho.
\end{align*}
Since $x_M - x_{M-1} = \omega(\rho)$ by \cref{eq:altered_grid}, the
above implies $x_j + \mu_j(b) \Delta t \leq Q$ for $\rho$ sufficiently small.
By a symmetric argument, we conclude that if $|x_j| < Q$, then $|x_j
\pm \mu_j(b) \Delta t| \leq Q$ for $\rho$ sufficiently small (i.e.,
there is no overstepping).

The $o(\sqrt{\rho})$ requirement in \cref{eq:altered_grid} is used to
ensure that the error from interpolation vanishes.
Indeed, if spatial grid points are not uniformly spaced, given a
smooth function $\varphi$ of time and space, a computation similar to
\cref{eq:sls_consistency_1} yields
\begin{multline*}
  \frac{
    \interp(
      \varphi^{n + 1},
      x_j + \mu_j(b) \Delta t
    ) - \varphi(n \Delta t, x_j)
  }{\Delta t}
  \\
  = \varphi_t(n \Delta t, x_j) + \mu(x_j, b) \varphi_x(n \Delta t, x_j)
  + O \left( \frac{\max_j \{x_{j+1} - x_j \}^2}{\Delta t} + \Delta t \right)
\end{multline*}
for $|x_j| < Q$.
Now, by the $o(\sqrt{\rho})$ requirement,
\[
  O \left( \frac{\max_j \{ x_{j + 1} - x_j \}^2}{\Delta t} + \Delta t \right)
  = o \left( \frac{(\sqrt{\rho})^2}{\rho} \right) + O(\rho)
  = o (1) \rightarrow 0.
\]

The above computations suggest that, unsurprisingly, the local order
of convergence at the boundaries is adversely affected by our new
choice of grid.
Therefore, it is worthwhile to point out that in the case of $\mu(-Q,
\cdot) \geq 0$ and $\mu(Q, \cdot) \leq 0$, we can use the original
uniform grid since no overstepping can occur.


\fi

\end{document}